\documentclass[11pt,reqno]{amsart}

\usepackage[utf8]{inputenc}
\usepackage[T1]{fontenc}
\usepackage{amsmath,amssymb,amsthm}
\usepackage{mathtools}
\usepackage[margin=1.1in]{geometry}
\usepackage{enumitem}
\usepackage{booktabs}
\usepackage{longtable}
\usepackage{array}
\usepackage{xcolor}
\usepackage[colorlinks=true,linkcolor=blue!50!black,citecolor=blue!50!black,urlcolor=blue!50!black]{hyperref}

\theoremstyle{plain}
\newtheorem{theorem}{Theorem}[section]
\newtheorem{lemma}[theorem]{Lemma}
\newtheorem{proposition}[theorem]{Proposition}
\newtheorem{corollary}[theorem]{Corollary}

\theoremstyle{definition}
\newtheorem{definition}[theorem]{Definition}

\theoremstyle{remark}
\newtheorem{remark}[theorem]{Remark}
\newtheorem{convention}[theorem]{Convention}

\newcommand{\N}{\mathbb{N}}
\newcommand{\Z}{\mathbb{Z}}
\newcommand{\Q}{\mathbb{Q}}

\newcommand{\B}{\mathbb{B}}
\newcommand{\Ndel}{\N_{\delta}}
\newcommand{\Zdel}{\Z_{\delta}}
\newcommand{\Qdel}{\Q_{\delta}}
\newcommand{\lean}[1]{\texttt{#1}}

\begin{document}

\title[The
$\delta$-calculus: from distinction to arithmetic]{The
$\delta$-calculus: from distinction to arithmetic}

\author{Jonathan Washburn}
\address{Recognition Physics Institute, Austin, Texas, USA}
\email{jon@recognitionphysics.org}

\author{Milan Zlatanovi\'c}
\address{Department of Mathematics, Faculty of Science and Mathematics,
University of Ni\v{s}, Vi\v{s}egradska 33, 18000 Ni\v{s}, Serbia}
\email{zlatmilan@yahoo.com}

\date{}

\begin{abstract} 
Let $\delta$ denote the primitive act of distinction, formally realized as the one-step extension $r\mapsto Sr$ of a finite record. We study the generated $\delta$-orbit and its arithmetic presentation $\mathbb{N}_\delta$. We prove that the $\delta$-orbit is initial among $\delta$-algebras: every $\delta$-algebra admits a unique structure-preserving map from the orbit. The map is injective when the successor operation is injective and the base point is not a successor. If the $\delta$-algebra also satisfies induction for all predicates, the map is bijective and gives the unique isomorphism with the generated $\delta$-orbit.

The corresponding $\delta$-calculus is an intuitionistic first-order proof system over the signature $\{0,S,+,\cdot\}$. {The name is independent of $\delta$-reduction in the $\lambda$-calculus: here $\delta$ generates the arithmetic carrier, and is not a term rewriting rule.} Each derivation carries a ledger recording the use of the law of excluded middle, the limited principle of omniscience, Markov's principle, and induction on quantified formulas. The last entry does not affect whether a derivation is forced. Every closed formula derivable in the forced fragment is true in the standard model.

Starting from $\delta$, we construct the choice-free number tower
$\delta \leadsto \mathbb{N}_\delta \hookrightarrow \mathbb{Z}_\delta \hookrightarrow \mathbb{Q}_\delta$. The metatheoretic systems $\mathbb{N}$, $\mathbb{Z}$, and $\mathbb{Q}$ each admit an explicit injection into $\mathbb{N}_\delta$.

We also classify the recognition quotients of $(\mathbb{N}_\delta,+,0)$. Under the law of excluded middle, every recognizer is either injective or has kernel congruence $\equiv_{i,p}$ for a unique pair $i\geq0$, $p\geq1$. In the noninjective case, the quotient is the finite monogenic monoid $M(i,p)$. We determine the exact {metatheoretic} logical price of this classification. A decidable congruence together with an explicit pair of distinct related elements implies the classification without additional nonconstructive principles. For a decidable congruence different from equality, Markov's principle is needed. For an arbitrary congruence, the dichotomy requires the law of excluded middle. The reverse implications show that the last two prices are exact. {These prices are computed in the ambient metatheory. They are distinct from the syntactic ledger of derivations in the $\delta$-calculus.} The main results are formalized in Lean 4.

\end{abstract}

\maketitle

\setcounter{tocdepth}{2}

\section{Introduction}\label{sec:intro}

{ Let \(\delta\) denote the primitive act of distinction.
{We identify the canonical arithmetic generated by iterated distinction, construct its choice-free number tower through the rationals, and determine the exact metatheoretic  logical cost of the recognition quotients of its additive monoid.}

The ambient metatheory provides propositions, equality, and collections of
elements, which we call {\it carriers.} The ambient metatheory is intuitionistic. It provides function types, finite
products, inductively generated carriers with their recursion and induction
principles, quotients by equivalence relations, and the natural numbers \(\N\),
which also index the de Bruijn variables (Section~\ref{sec:calculus}).
Equality and order on \(\N\) are decidable. {We work in a fixed universe \(\mathrm{Type}_u\). A carrier is a type \(K:\mathrm{Type}_u\), and a predicate on \(K\) is a map \(P:K\to\mathrm{Prop}\). Equality is the ambient equality, and quotients are formed by equivalence relations on carriers. In Theorem~\ref{thm:rigidity}, induction applies to all such predicates, including the image predicate used for surjectivity.
} We do not use the law of excluded middle, the limited principle of
omniscience, Markov's principle, or the axiom of choice in the metatheory,
except where explicitly assumed, as in
Section~\ref{sec:classificationledger}. The ledger (Section~\ref{ssec:ledger}) records principles
used inside the $\delta$- calculus.  

The purpose of this construction is not to obtain another copy of the natural numbers, since the metatheory already contains \(\mathbb N\). The generated \(\delta\)-orbit is instead the canonical object obtained from the empty record by repeated one-step extension. By Theorem~\ref{thm:rigidityheadline}, it is initial among \(\delta\)-algebras, and the Peano conditions characterize it up to a unique isomorphism. Thus the resulting arithmetic is unique up to isomorphism.

Each element of the orbit records a finite number of successive distinction steps, and the later stages of the number tower are constructed explicitly from the generated carrier by group completion and fraction construction.
 
The main question is therefore not what the natural numbers are, but what can be constructed from the generated carrier and at what logical price.  For the recognition quotients of \((\Ndel,+,0)\), this price is exact: a decidable congruence together with an explicit pair of distinct related elements implies the classification without additional nonconstructive principles; decidability alone requires Markov's principle; and an arbitrary congruence requires excluded middle.

\begin{definition}\label{def:floor}
Let \(K\) be a carrier. Write
\[
\operatorname{Dist}(K)
\;:\Longleftrightarrow\;
\exists x,y\in K,\quad x\neq y.
\]
A proof \(h:\operatorname{Dist}(K)\) is called a \emph{distinction witness} on
\(K\).
\end{definition}

The predicate \(\operatorname{Dist}(K)\) uses no additional structure on \(K\) beyond the equality of the ambient metatheory. A carrier \(K\) is called nontrivial if
\(\operatorname{Dist}(K)\) holds. \(\operatorname{Dist}(K)\) gives two unequal elements, but it gives no operation which can be iterated.
{The predicate is not used to generate the \(\delta\)-orbit. The symbol \(\delta\) is only the name of the one-step extension \(r\mapsto Sr\) in Definition~\ref{orbit}. A distinction witness does not produce \(0\), \(S\), or induction.}

{The word \emph{distinction} has  formal and philosophical uses
\cite{lof,wengert}. Spencer-Brown's \emph{Laws of Form}~\cite{lof} develops
a calculus from a mark of distinction. Our construction is different. The
primitive is not a mark in a plane, but the one-step extension of a finite
record. The generated object is the \(\delta\)-orbit. The condensation rules
of~\cite{lof} absorb repeated marking, whereas each application of \(S\)
produces a further stage of the generated record. Successor nonzero,
successor injectivity, and induction then give the natural-number object.

A related physical idea appears in Wheeler's thesis \emph{it from bit}, where
particles, fields, and spacetime receive their physical meaning from elementary
yes/no registrations~\cite{wheeler1992}. Since each such registration
presupposes two distinguishable alternatives, here we study the prior
mathematical step: the finite iteration of distinction. We do not use the calculus of~\cite{lof}, and
the scholastic doctrine of formal distinction~\cite{wengert} is not part of
the construction.

In constructive algebra, a positive apartness relation is often used in place
of negative inequality~\cite{crvmitrom2013,crvmitrom2016,mhb2022}. No
apartness relation is assumed here.} {The ambient host is an intuitionistic type theory in the sense of Martin-Löf~\cite{mltt}, close to the Calculus of Constructions~\cite{coquandhuet}. We do not introduce a new type theory. Constructive reverse mathematics~\cite{ishihara,diener} asks which principles suffice for a theorem, the ledger of Section~\ref{ssec:ledger} instead records what a given derivation uses.
}
\begin{definition}\label{orbit}
The \emph{\(\delta\)-orbit} \(\operatorname{Orb}(0,S)\) is the collection of finite records generated
by two formation rules:\[
\frac{}{\;0\ \text{is a record}\;},
\qquad
\frac{\;r\ \text{is a record}\;}{\;S r\ \text{is a record}\;} .
\]
Here \(0\) is the empty record, and \(S r\) is the record obtained from \(r\) by
one further act of distinction.
\end{definition}

\begin{proposition}\label{prop:orbitlaws}
For all \(r,q\in\operatorname{Orb}(0,S)\),
\[
Sr\neq0,
\qquad
Sr=Sq\longrightarrow r=q.
\]
Moreover, if a predicate \(P\) on \(\operatorname{Orb}(0,S)\) satisfies
\(P(0)\) and
\[
P(r)\longrightarrow P(Sr)
\quad\text{for every }r,
\]
then \(P(r)\) holds for every \(r\in\operatorname{Orb}(0,S)\).
\end{proposition}

\begin{proof}
The orbit is inductively generated, so it carries the recursion principle:
for every carrier \(K\), every \(k\in K\) and every
\(f\colon\operatorname{Orb}(0,S)\times K\to K\) there is a map
\(g\colon\operatorname{Orb}(0,S)\to K\) with \(g(0)=k\) and
\(g(Sr)=f(r,g(r))\).

For \(Sr\neq0\), apply the recursion principle to
\(K=\B=\{0,1\}\), with \(k=0\) and \(f(r,z)=1\).
Then \(g(0)=0\) and \(g(Sr)=1\), so \(Sr=0\) would imply \(0=1\).

{
For injectivity, let us define
\(\operatorname{pred}\colon\operatorname{Orb}(0,S)\to
\operatorname{Orb}(0,S)\) by recursion with \(k=0\) and \(f(r,z)=r\), so that
\(\operatorname{pred}(Sr)=r\) for every \(r\). Hence \(Sr=Sq\) gives
\[
r=\operatorname{pred}(Sr)=\operatorname{pred}(Sq)=q.
\]}

The last statement is the structural induction principle of the inductively generated orbit.
\end{proof}

The orbit is not defined as the subset
\(
\{S^n0:n\in\N\}
\)
of a pre-existing carrier, so its definition does not use a prior
natural-number object as an index. The notation \(S^n0\) is used only
metatheoretically, for the result of applying \(S\) \(n\) times to \(0\). The
successor laws and the induction principle for the generated orbit are proved
in Proposition~\ref{prop:orbitlaws}.

{The rigidity theorem (Theorem~\ref{thm:rigidityheadline}) reverses the usual direction of explanation. The generated $\delta$-orbit is not merely another model of the Peano conditions. It is the object those conditions describe. Initiality sends the orbit canonically into every $\delta$-algebra. When the successor map is injective and the base point is not a successor, this canonical map is an embedding, so every such faithful realization already contains a copy of the orbit. Induction makes the map surjective, and initiality makes the resulting isomorphism unique. The Peano conditions  do not create an arithmetic structure by stipulation. They characterize, up to a unique structure-preserving relabeling, the arithmetic generated by repeated distinction. This is a representation-invariance statement within an ambient metatheory supporting inductive generation.}

Initiality and Peano categoricity are classical~\cite{lawvere}. Here we combine them with distinction-based generation, the choice-free number tower, and the exact {metatheoretic logical prices} of recognition quotients.




\medskip
The \(\delta\)-calculus is an intuitionistic first-order proof system over the signature \(\{0,S,+,\cdot\}\), defined in Section~\ref{sec:calculus}. The symbol \(\delta\) is not part of the formal language. {In particular, \(\delta\) is not the \(\delta\)-rule of the \(\lambda\)-calculus~\cite{barendregt}.} The calculus is called the \(\delta\)-calculus because its symbols \(0\) and \(S\), the distinction axioms, and the induction rule come from the generated \(\delta\)-orbit.
In the arithmetic projection (Definition \ref{def:arithprojection}), the empty record is written as \(0\), and the
one-step extension of a record is written by the successor symbol \(S\). At the first-order level, the successor theorems become the distinction axioms
\[
S\,t\neq 0,\qquad S\,t=S\,s\longrightarrow t=s.
\]
They are validated in the standard model \(\mathfrak N=(\mathbb N,0,S,+,\cdot)\). The operations \(+\) and \(\cdot\)
are governed by their recursion equations.

Let \(\Ndel\) denote the arithmetic structure carried by the generated \(\delta\)-orbit, with \(+\) and \(\cdot\) defined by recursion. Its successor laws and induction are given by Proposition~\ref{prop:orbitlaws}.

We fix the generation grammar consisting of the generation of \(\Ndel\) from
the primitive act \(\delta\), followed by group completion and fraction
construction. The stages of the construction are
\[
\delta \leadsto \Ndel
\xhookrightarrow{\iota_{\Zdel}} \Zdel
\xhookrightarrow{\iota_{\Qdel}} \Qdel .
\]

The first arrow denotes the inductive generation of the \(\delta\)-orbit of
finite records and the natural-number object presented by it. 
The next stage is
the group completion
\[
\Zdel :=(\Ndel\times \Ndel)/{\sim},
\qquad
(a,b)\sim(c,d)
\Longleftrightarrow
a+d=c+b .
\]
The map \(\iota_{\Zdel}\) is the canonical embedding of \(\Ndel\) into this
quotient, given by \(a\mapsto [(a,0)]\).

The next stage is the fraction construction \(\Qdel\), presented by
ratio orbits with signed-orbit numerators and nonzero denominators in
\(\Ndel\), as defined in Section~\ref{sec:tower}. The canonical embedding
\[
\iota_{\Qdel}\colon\Zdel\longrightarrow\Qdel
\]
is induced by adjoining denominator \(1\).

Thus the integer and rational stages are defined by their internal quotient
relations. The comparison maps into the standard systems \(\mathbb Z\) and
\(\mathbb Q\) are constructed  after these quotients are formed.  The construction is choice-free. In the \(\delta\)-calculus, each derivation carries a ledger recording the nonconstructive principles used. The main
results are verified in Lean~4 (Section~\ref{sec:lean}).

\begin{definition}\label{def:encodable}
For a carrier \(X\), write
\[
\operatorname{Enc}_\delta(X)
\;:\Longleftrightarrow\;
\text{there exists an injection } X\hookrightarrow \Ndel .
\]
An object satisfying this predicate is called {\it \(\delta\)-encodable.}
\end{definition}

 We prove that
the metatheoretic number systems \(\mathbb N\), \(\mathbb Z\), and \(\mathbb Q\)
are \(\delta\)-encodable.
In this paper, we stop the tower at \(\Qdel\). The real-number stage is not treated in this paper. Its construction gives a further constructive problem, since the Cauchy and Dedekind constructions are not always
equivalent.

\subsection{Main results}\label{ssec:main}
{
{The paper has three main results. The first is the initiality and rigidity
of the generated $\delta$-orbit. The second is the construction of the number tower from the generated \(\delta\)-orbit. The third is the classification of the recognition quotients of \((\Ndel,+,0)\), together with {their exact metatheoretic logical prices}.}

The third result contains the new quantitative part of the paper. The classical index-period classification of congruences on a monogenic monoid~\cite{birkhoff,cliffordpreston} is refined by assigning exact logical prices to its three forms. The reverse implications of Corollary~\ref{cor:exact} show that the Markov and excluded-middle assumptions cannot be weakened. The three main results are naturally ordered: the first identifies the generated carrier, the second constructs its choice-free number tower, and the third gives the exact { metatheoretic} logical cost of classifying its recognition quotients.

We use the following standard universal property of the term syntax with de Bruijn variables~\cite{debruijn}.

\begin{proposition}\label{prop}
For every algebra \(A\) of the signature \(\{0,S,+,\cdot\}\), with carrier
\(|A|\), and every assignment
\(
\rho\colon\mathbb N\to |A|,
\)
there exists a unique homomorphism from the term algebra to \(A\) extending
\(\rho\). Thus the term syntax is the free algebra of this signature on the
de Bruijn variables.
\end{proposition}}

Here \(\mathbb N\) is  the metatheoretic index set for the de Bruijn
variables.
The recursion equations for \(+\) and \(\cdot\) are validated in the standard model
\[
\mathfrak{N}=(\mathbb N,0,S,+,\cdot),
\]
with the usual
operations. The resulting
arithmetic is developed in Section~\ref{ssec:b2}.     

 {\begin{theorem}\label{thm:rigidityheadline}
Let $A=(|A|,0_A,S_A)$ be a $\delta$-algebra. There exists a unique
homomorphism
\(
\operatorname{rec}_A\colon\Ndel\longrightarrow A.
\)
If $S_A$ is injective and $0_A\notin\operatorname{im}(S_A)$, then
$\operatorname{rec}_A$ is injective. If, in addition, $A$ satisfies
induction for all predicates, then $\operatorname{rec}_A$ is bijective
and is the unique $\delta$-algebra isomorphism from $\Ndel$ to $A$.
\end{theorem}

Theorem~\ref{thm:rigidityheadline} follows from
Theorems~\ref{thm:initiality}, \ref{thm:embedding},
and~\ref{thm:rigidity}, proved in Section~\ref{sec:rigidity}.
}
\begin{theorem}\label{thm:construction}
The following hold by choice-free constructions.
\begin{enumerate}[label=(\alph*)]
 \item The carriers \(\Ndel\), \(\Zdel\), and \(\Qdel\) are constructed from
the generated \(\delta\)-orbit. The carrier \(\Ndel\) is the natural-number
object presented by this orbit, \(\Zdel\) is the group completion of
\(\Ndel\), and \(\Qdel\) is the field of fractions of \(\Zdel\), presented
by ratio orbits with signed-orbit numerators. 

\item The metatheoretic number systems \(\mathbb N\), \(\mathbb Z\), and
\(\mathbb Q\) are \(\delta\)-encodable.
\end{enumerate}
\end{theorem}

 The classification of the
recognition quotients of \((\Ndel,+,0)\) has three forms, depending on the
assumptions imposed on the congruence.

\begin{theorem}\label{thm:pricing}
Let \(c\) be a congruence on \((\Ndel,+,0)\). The index-period classification
has the following three forms:
\begin{enumerate}[label=(\alph*)]
\item if \(c\) is decidable and an explicit pair of distinct related elements
is given, then
\(
c=\equiv_{i,p}
\)
for a unique pair \(i\geq0\), \(p\geq1\), without using EM, LPO, or MP.

\item If \(c\) is decidable and is not equality, then the same conclusion is
conditional on \(\mathrm{MP}\).

\item For an arbitrary congruence \(c\), it is conditional on \(\mathrm{EM}\) that \(c\) is either equality or
\(
c=\equiv_{i,p}
\)
for a unique pair \(i\geq0\), \(p\geq1\).
\end{enumerate}
Conversely, the general statements in \((b)\) and \((c)\) imply \(\mathrm{MP}\) and \(\mathrm{EM}\), respectively. Thus these two {metatheoretic logical prices} cannot be lowered. 
\end{theorem}

{The three prices have a simple meaning. If a pair of distinct related elements is given, a bounded search finds the index and the period. No additional principle is used. If we know only that the congruence is not equality, a witness pair is available only under double negation, and Markov's principle extracts it. If the congruence is arbitrary, even the alternative "equality or not" is missing, so excluded middle is needed. The reverse implications show that the last two steps cannot be dropped.
} 
The prices in Theorem~\ref{thm:pricing} are metatheoretic, in the sense
fixed in Convention~\ref{conv:classificationprice}.
\subsection{Organization}\label{ssec:org} The paper is organized as follows. {Section~\ref{sec:rigidity} proves the initiality, faithful embedding, and
rigidity of the generated $\delta$-orbit
(Theorem~\ref{thm:rigidityheadline}).}  Section~\ref{sec:calculus} defines the \(\delta\)-calculus and the ledger, and proves the soundness of the forced fragment (Theorem~\ref{thm:sound}). Section~\ref{sec:mechanism}
proves the universal property of the term algebra (Proposition~\ref{prop}) and
develops the arithmetic validated in the standard model.  
Section~\ref{sec:tower} proves Theorem~\ref{thm:construction} by constructing the number tower
\(
\Ndel\hookrightarrow\Zdel\hookrightarrow\Qdel
\)
and giving explicit encodings of \(\mathbb N\), \(\mathbb Z\), and \(\mathbb Q\) into \(\Ndel\).
Section~\ref{sec:recognition} classifies the
recognition quotients of the additive monoid \((\Ndel,+,0)\)
(Theorems~\ref{thm:catalogue} and~\ref{thm:recognition}). Section~\ref{sec:classificationledger} prices this classification {in the ambient metatheory} and proves Theorem~\ref{thm:pricing}. It gives the three logical forms and proves that the Markov and excluded-middle prices cannot be lowered.
Section~\ref{sec:lean} describes the Lean formalization and its axiom audit,
and Section~\ref{sec:conclusion} contains concluding remarks.

One limitation of the present carrier is structural, and it also indicates the direction of the sequel. The additive monoid \((\Ndel,+,0)\) is free on one generator: it records the number of successive distinction steps, but not their order or type. Thus the recognition quotients classified here describe the simplest generated history carrier. In the sequel, ordered histories are represented by monoids of closed walks on a graph, where different loops may lead to noncommutative composition. The exact logical-pricing method of Section~\ref{sec:classificationledger} then provides the starting point for the corresponding recognition problem.

{
\section{Initiality and rigidity of the
\texorpdfstring{$\delta$}{delta}-orbit}
\label{sec:rigidity}

We now state the universal property of the generated $\delta$-orbit. $\Ndel$ denotes the $\delta$-orbit of
Definition~\ref{orbit}, with constructors $0$ and $S$.

\begin{definition}\label{def:dalgebra}
A \emph{$\delta$-algebra} is a triple
\(
A=(|A|,0_A,S_A),
\)
where $|A|$ is a carrier, $0_A\in|A|$, and
\(
S_A\colon |A|\longrightarrow |A|
\)
is a unary operation. A homomorphism \(f\colon A\to B\) of \(\delta\)-algebras is a map
\(f\colon|A|\to|B|\) such that
\[
f(0_A)=0_B\quad\text{and}\quad
f(S_Ax)=S_Bf(x)
\]
for every \(x\in|A|\).
\end{definition}

The orbit $(\Ndel,0,S)$ is a $\delta$-algebra. Proposition~\ref{prop} states the freeness of the term algebra over the
signature $\{0,S,+,\cdot\}$ with variables. Here we study the generated $\delta$-orbit as an algebra with operations $0$ and $S$.

\begin{theorem}\label{thm:initiality}
For every $\delta$-algebra $A$, there exists a unique homomorphism
\(
\operatorname{rec}_A\colon\Ndel\longrightarrow A.
\)
It is determined by
\[
\operatorname{rec}_A(0)=0_A\quad\mbox{and}\quad
\operatorname{rec}_A(Sn)
=
S_A(\operatorname{rec}_A(n)).
\]
\end{theorem}

\begin{proof}
Existence follows from the recursion principle of the inductively
generated $\delta$-orbit.

Let $f\colon\Ndel\to A$ be a homomorphism. We prove
\(
f(n)=\operatorname{rec}_A(n)
\)
by structural induction on $n$. For $n=0$, we have
\(
f(0)=0_A=\operatorname{rec}_A(0).
\)
If the equality holds for $n$, then
\[
f(Sn)
=
S_Af(n)
=
S_A\operatorname{rec}_A(n)
=
\operatorname{rec}_A(Sn).
\]
Hence $f=\operatorname{rec}_A$.
\end{proof}

{\begin{theorem}\label{thm:embedding}
Let \(A\) be a \(\delta\)-algebra. Suppose that \(S_A\) is injective and
\(0_A\notin\operatorname{im}(S_A)\). Then the canonical homomorphism
\(
\operatorname{rec}_A\colon\Ndel\longrightarrow A
\)
is injective.
\end{theorem}}
\begin{proof}
We prove
\[
\operatorname{rec}_A(a)=\operatorname{rec}_A(b)
\quad\Longrightarrow\quad
a=b
\]
by structural induction on \(a\).

Let $a=0$.  If $b=0$, the claim holds. If $b=Sb'$, then
\[
0_A
=
\operatorname{rec}_A(0)
=
\operatorname{rec}_A(Sb')
=
S_A(\operatorname{rec}_A(b')).
\]
This contradicts the assumption \(0_A\notin\operatorname{im}(S_A)\).

Let $a=Sa'$. The element $b$ cannot be $0$, since otherwise
\[
S_A(\operatorname{rec}_A(a'))
=
\operatorname{rec}_A(Sa')
=
\operatorname{rec}_A(0)
=
0_A.
\]
Thus $b=Sb'$ for some $b'$. Then
\[
S_A(\operatorname{rec}_A(a'))
=
S_A(\operatorname{rec}_A(b')).
\]
Since $S_A$ is injective,  
\(
\operatorname{rec}_A(a')
=
\operatorname{rec}_A(b').
\)
The induction hypothesis gives $a'=b'$, and therefore $a=b$.
\end{proof}

Thus, if \(S_A\) is injective and \(0_A\notin\operatorname{im}(S_A)\), the canonical homomorphism
\[
\operatorname{rec}_A\colon\Ndel\to A
\]
identifies \(\Ndel\) with the sub-\(\delta\)-algebra \(\operatorname{im}(\operatorname{rec}_A)\) of \(A\).

\begin{definition}\label{def:peano}
A $\delta$-algebra $A$ is a \emph{Peano realization} if

\begin{enumerate}
\item $S_A$ is injective;
\item $S_Ax\neq0_A$ for every $x\in|A|$;
\item for every predicate $P$ on $|A|$, if
\[
P(0_A)\quad\mbox{and}\quad
P(x)\longrightarrow P(S_Ax)
\]
for every $x\in|A|$, then $P(x)$ holds for every $x\in|A|$.
\end{enumerate}
\end{definition}

\begin{theorem}\label{thm:rigidity}
For every Peano realization $A$, the canonical homomorphism
\(
\operatorname{rec}_A\colon\Ndel\longrightarrow A
\)
is bijective. It is  the unique $\delta$-algebra isomorphism
\(\Ndel\cong A\).
\end{theorem}

\begin{proof}
Injectivity follows from Theorem~\ref{thm:embedding}. For surjectivity, let us
consider the predicate
\[
P(x)
\;:\Longleftrightarrow\;
\exists n\in\Ndel,\quad
\operatorname{rec}_A(n)=x.
\]
The base case holds with witness $0$, since
\(
\operatorname{rec}_A(0)=0_A.
\) Suppose that $P(x)$ holds. Let $n\in\Ndel$ satisfy
\[
\operatorname{rec}_A(n)=x.
\]
Then
\[
\operatorname{rec}_A(Sn)
=
S_A(\operatorname{rec}_A(n))
=
S_Ax.
\]
Thus $Sn$ witnesses $P(S_Ax)$. Induction in $A$ gives $P(x)$ for every
$x\in|A|$. Hence $\operatorname{rec}_A$ is surjective.

The uniqueness of the isomorphism follows from
Theorem~\ref{thm:initiality}.
\end{proof}
{The predicate \(P\) used in the previous theorem is an ambient predicate. The third condition of
Definition~\ref{def:peano} applies to every such predicate, including the
image predicate used for surjectivity.}

\begin{remark}\label{rem:displayinstance}
Theorem~\ref{thm:rigidity} is relative to an ambient metatheory supporting
inductive generation. It shows that every Peano realization is uniquely isomorphic to \(\Ndel\)
(cf.~\cite{lawvere}). For the standard Peano algebra
\[
A=(\N,0,n\mapsto n+1),
\]
the canonical map \(\operatorname{rec}_A\) is the counting map \(D\) of
Section~\ref{sec:tower}. Hence Lemma~\ref{lem:display} is the
standard-model instance of Theorem~\ref{thm:rigidity}.
\end{remark}

Theorems~\ref{thm:initiality}, \ref{thm:embedding}, and
\ref{thm:rigidity} together prove
Theorem~\ref{thm:rigidityheadline}.
}

\section{The \texorpdfstring{$\delta$}{\delta}-calculus}\label{sec:calculus}

We now define the \(\delta\)-calculus, a natural-deduction system for an
intuitionistic first-order arithmetic with equality over the signature
\(\{0,S,+,\cdot\}\).
Each derivation has a \emph{ledger} recording the use of the law of excluded
middle (EM), the limited principle of omniscience (LPO), Markov's principle
(MP), and induction on formulas containing a quantifier. 
The
ledger is defined in Section~\ref{ssec:ledger}.

Terms are generated by de Bruijn variables, zero, successor,
addition, and multiplication:
\[
t ::= x_n \mid 0 \mid S\,t \mid t+t \mid t\cdot t .
\]
{Under the arithmetic projection (Definition~\ref{def:arithprojection}), the symbol \(0\) represents the empty record and \(S\) represents one further extension of a record.} 

Formulas are generated by the grammar of first-order logic with
equality~\cite{vandalen}:
\[
\varphi ::= t=t \mid \bot \mid \varphi\wedge\varphi
\mid \varphi\vee\varphi
\mid \varphi\rightarrow\varphi
\mid \forall\,\varphi
\mid \exists\,\varphi ,
\]
where \(\neg\varphi := \varphi\rightarrow\bot\).

The nonlogical vocabulary of the calculus is exactly \(0,S,+,\cdot\).  Lifting and substitution
are defined by the standard structural recursions for de Bruijn syntax. For
every metatheoretic natural number \(n\), the expression \(S^n0\) is a closed
term, called {\it the numeral} for \(n\).

\subsection{The ledger}\label{ssec:ledger}

Let
\[
\mathcal{L}=\{\mathrm{EM},\mathrm{LPO},\mathrm{MP},\mathrm{QInd}\}.
\]
A \emph{ledger} \(\Lambda\) is a subset of \(\mathcal{L}\).  For each rule instance \(r\), define a label \(\ell(r)\subseteq\mathcal L\) as
follows: EM, LPO, and MP contribute their corresponding singleton. Induction on
a formula containing a quantifier contributes \(\{\mathrm{QInd}\}\). For every other rule instance \(r\), we set
\(
\ell(r)=\varnothing.
\) 

 The
\emph{ledger} of a derivation \(d\) is
\[
\Lambda(d)=\bigcup_{r\ \mathrm{in}\ d}\ell(r).
\]
A derivation has \emph{empty ledger} if \(\Lambda(d)=\varnothing\). Let
\[
\Lambda_{\mathrm{pos}}(d)
=
\Lambda(d)\cap\{\mathrm{EM},\mathrm{LPO},\mathrm{MP}\}.
\]
A derivation \(d\) is \emph{forced} if
\(\Lambda_{\mathrm{pos}}(d)=\varnothing\), and \emph{conditional on \(O\)}, for
\(O\subseteq\{\mathrm{EM},\mathrm{LPO},\mathrm{MP}\}\), if
\(\Lambda_{\mathrm{pos}}(d)=O\). The entry \(\mathrm{QInd}\) does not affect
whether \(d\) is forced.


 The definition of the ledger is close in spirit to constructive reverse
mathematics, which classifies theorems by the principles they
require~\cite{bishop,ishihara}. Constructive reverse
mathematics studies the relation of theorems to principles over a fixed base
theory. Its standard families also include principles such as the weak limited principle of omniscience (WLPO) and the lesser limited principle of omniscience (LLPO), whose relations with
\(\mathrm{LPO}\) and \(\mathrm{MP}\) are not linear. Our ledger is instead a
syntactic annotation carried by each derivation of the
\(\delta\)-calculus. The two are complementary: the ledger records what a
given derivation uses, while reverse mathematics studies which principles are
sufficient or necessary for a theorem. We record \(\mathrm{EM}\),
\(\mathrm{LPO}\), and \(\mathrm{MP}\), since these are the principles used by
the present derivations. The principles \(\mathrm{WLPO}\) and
\(\mathrm{LLPO}\) are not needed here and are not part of the ledger.

\subsection{Derivations}\label{ssec:checker}

The \(\delta\)-calculus is presented as a natural deduction system with the
following rules:
\begin{itemize}
\item the assumption rule, with formulas in the context accessed by de Bruijn indices;
\item the standard rules for equality, including reflexivity and Leibniz substitution;
\item the first-order distinction axioms
\[
S\,t\neq0,\qquad S\,t=S\,s\rightarrow t=s;
\]
\item the recursion equations
\[
t+0=t,\qquad
t+S\,s=S(t+s),
\]
\[
t\cdot0=0,\qquad
t\cdot S\,s=t\cdot s+t;
\]
\item the induction rule;
\item the intuitionistic rules for propositional connectives and quantifiers;
\item the additional schemas corresponding to EM, LPO, and MP.
\end{itemize}
The induction rule is
\[
\frac{
\Gamma\vdash\varphi[0]
\qquad
\Gamma\vdash\forall\bigl(\varphi\rightarrow\varphi[Sx_0]\bigr)
}{
\Gamma\vdash\forall\varphi
}
\;(\mathrm{Ind}),
\]
where \(\varphi[t]\) denotes substitution of the term \(t\) for the free variable \(x_0\), with the usual lifting of the remaining de Bruijn variables. The rule contributes \({\mathrm{QInd}}\) when \(\varphi\) contains a quantifier.

{ A context \(\Gamma\) is a finite list of formulas, and the derivability judgment is  \(\Gamma\vdash\varphi\). The assumption rule selects a formula of \(\Gamma\) by its de Bruijn index.

Lifting and substitution are defined by the standard capture-avoiding
structural recursion for de Bruijn syntax~\cite{debruijn}. Under a binder,
the substitution is lifted so that the newly bound variable is fixed. Equality uses reflexivity, symmetry, transitivity, Leibniz substitution, and congruence. The usual intuitionistic introduction and elimination rules are used for connectives and quantifiers.

The schemas \(\mathrm{EM}\), \(\mathrm{LPO}\), and \(\mathrm{MP}\) can be applied in any well-formed context. Each application contributes the corresponding entry to the ledger. For \(\mathrm{LPO}\) and \(\mathrm{MP}\), the matrix \(\theta\) must be decidable in the forced fragment. Lemma~\ref{lem:qfdec} gives this for the quantifier-free case used below.
}

Let \(\varphi\) be a formula, and let \(\theta(x,\vec y)\) be a
\emph{decidable} formula, in the sense that
\[
\forall x\,\bigl(\theta(x,\vec y)\vee\neg\theta(x,\vec y)\bigr)
\]
is derivable in the forced fragment. {Lemma~\ref{lem:qfdec} below shows that every
quantifier-free formula is decidable in this sense.

The additional schemas are
\[
\varphi\vee\neg\varphi \qquad (\mathrm{EM}),
\]
\[
\exists x\,\theta(x,\vec y)\vee
\forall x\,\neg\theta(x,\vec y) \qquad (\mathrm{LPO}).
\]
This is the arithmetic form of \(\mathrm{LPO}\), with decidable predicates on \(\mathbb N\) in place of binary sequences~\cite{bishop,ishihara}.

\[
\neg\neg\exists x\,\theta(x,\vec y)
\longrightarrow
\exists x\,\theta(x,\vec y) \qquad (\mathrm{MP}).
\]
In the \(\delta\)-calculus, the following implications hold
\[
\mathrm{EM}\Longrightarrow\mathrm{LPO}\Longrightarrow\mathrm{MP}.
\]
Applying \(\mathrm{EM}\) to
\(\exists x\,\theta(x,\vec y)\) gives \(\mathrm{LPO}\). Under
\(\mathrm{LPO}\), the alternative
\(\forall x\,\neg\theta(x,\vec y)\) contradicts
\(\neg\neg\exists x\,\theta(x,\vec y)\), and therefore gives
\(\mathrm{MP}\).}

{\begin{lemma}\label{lem:qfdec}
Every quantifier-free formula is decidable without using
\(\mathrm{EM}\), \(\mathrm{LPO}\), or \(\mathrm{MP}\). More precisely, if
\(\theta(\vec x)\) is quantifier-free, then

$$
\forall \vec x\,\bigl(\theta(\vec x)\vee\neg\theta(\vec x)\bigr)
$$
has a derivation whose ledger contains none of
\(\mathrm{EM}\), \(\mathrm{LPO}\), or \(\mathrm{MP}\).
\end{lemma}

\begin{proof}
We first prove decidability of equality:
$$
\forall x\forall y\,(x=y\vee x\neq y).
$$
We use induction on \(x\), with a subsidiary induction on \(y\). The base
and successor cases follow from
$$
S t\neq 0
\qquad\text{and}\qquad
S t=S s\longrightarrow t=s.
$$
The subsidiary induction formulas are quantifier-free. The outer induction
formula contains the quantifier over \(y\), and therefore contributes
\(\mathrm{QInd}\), but no instance of
\(\mathrm{EM}\), \(\mathrm{LPO}\), or \(\mathrm{MP}\).

The general statement follows by structural induction on
\(\theta\). Atomic equalities are decidable by the preceding argument, and
\(\bot\) is decidable intuitionistically. If \(\varphi\) and \(\psi\) are
decidable, then so are
$$
\varphi\wedge\psi,\qquad
\varphi\vee\psi,\qquad
\varphi\rightarrow\psi,\qquad
\neg\varphi,
$$
by case analysis on the corresponding decision disjunctions. Hence every
quantifier-free formula is decidable without using any schema from
\(\{\mathrm{EM},\mathrm{LPO},\mathrm{MP}\}\).
\end{proof}
}


Terms are interpreted in the standard arithmetic structure
\(
\mathfrak N=(\mathbb N,0,S,+,\cdot),
\)
where the operations are the usual operations on the metatheoretic natural
numbers. Formulas are evaluated by the corresponding Tarski semantics.

\begin{theorem}\label{thm:sound}
If a forced derivation in the empty context proves a closed formula
\(\varphi\), then
\[
\mathfrak N\models\varphi .
\]
\end{theorem}
\begin{proof} We prove by induction that, for every context
\(\Gamma\), valuation \(\rho\), and forced derivation of
\(\Gamma\vdash\varphi\), satisfaction of \(\Gamma\) under \(\rho\) implies
satisfaction of \(\varphi\) under \(\rho\).

A valuation \(\rho:\N\to\N\) interprets the de Bruijn variable \(x_n\) as
\(\rho(n)\). Under a quantifier it is extended by
\[
(a::\rho)(0)=a,
\qquad
(a::\rho)(n+1)=\rho(n).
\]
A context is satisfied if all its formulas are satisfied under the given
valuation. The equality and logical rules follow from the corresponding Tarski
semantics, while the distinction axioms and the recursion equations are valid in
\(\mathfrak N\). The quantifier cases use the standard lifting and substitution
lemmas for de Bruijn syntax~\cite{debruijn}. The induction-rule case follows from induction in
\(\mathbb N\). 

Since \(d\) is forced, the cases for \(\mathrm{EM}\), \(\mathrm{LPO}\), and
\(\mathrm{MP}\) do not occur. The proof uses none of these principles and no
axiom of choice.
\end{proof}

{}

\section{The universal property and arithmetic}\label{sec:mechanism}
This section proves the universal property stated in
Proposition~\ref{prop}. It also develops the arithmetic validated in the
standard model.

Every term is generated by the constructors
\[
x_n,\qquad 0,\qquad S,\qquad +,\qquad \cdot,
\]
and every term has a unique such constructor form. At this stage the term
algebra is not quotiented by the recursion equations.

Let \(A\) be an algebra of the signature \(\{0,S,+,\cdot\}\), with carrier
\(|A|\): a constant \(0_A\in|A|\), a unary operation
\[
S_A:|A|\to|A|,
\]
and binary operations
\[
+_A,\ \cdot_A:|A|\times|A|\to|A| .
\]
Let
\[
\rho:\mathbb N\to|A|
\]
be an assignment of the de Bruijn variables. No equations are assumed in
\(A\).

\begin{proof}[Proof of Proposition~\ref{prop}]
Define the map \(h\) on terms by structural recursion:
\[
h(x_n)=\rho(n),\qquad h(0)=0_A,\qquad h(S\,t)=S_A\,h(t),
\]
\[
h(t+s)=h(t)+_A h(s),\qquad h(t\cdot s)=h(t)\cdot_A h(s).
\]
These equations define a homomorphism from the term algebra to \(A\). By construction it
preserves \(0,S,+,\cdot\) and extends \(\rho\), which gives
existence.

For uniqueness, let \(h'\) be any homomorphism from the term algebra to \(A\) such that
\(h'(x_n)=\rho(n)\) for all \(n\). A structural induction on \(t\) shows that \(h'(t)=h(t)\) for every term
\(t\). Hence \(h'=h\).\end{proof}

\subsection{Arithmetic  and derived objects}\label{ssec:b2}

In the standard model \(\mathfrak N\), the closed numeral \(S^n0\) is
interpreted as \(n\).

The recursion equations for \(+\) and \(\cdot\) are validated in
\(\mathfrak N\). They do not follow from the universal property, since an arbitrary algebra of
the signature is not assumed to satisfy any equations. The identities for \(+\) and \(\cdot\), including associativity,
commutativity, and distributivity, are derived with empty ledger. Indeed,
each of them is proved by induction on a formula which is quantifier-free in
the induction variable, the remaining variables being free, so no instance of
the induction rule contributes \(\mathrm{QInd}\). The universal closures are
then obtained by \(\forall\)-introduction, which contributes nothing.


{
The forced fragment (Section~\ref{ssec:ledger})  is an intuitionistic
first-order arithmetic over the signature \(\{0,S,+,\cdot\}\), with equality,
the distinction axioms, the recursion equations, and full induction. Empty-ledger derivations also exclude  \(\mathrm{QInd}\). Hence induction in that
fragment is restricted to quantifier-free formulas.}

\section{The number tower}\label{sec:tower}

This section proves Theorem~\ref{thm:construction}$(a)$. The construction proceeds from the orbit carrier \(\Ndel\) to its group completion \(\Zdel\), and then to the fraction construction \(\Qdel\). All constructions in this section are choice-free.
\medskip


The natural-number object \(\Ndel\) is presented by the generated
\(\delta\)-orbit of Definition~\ref{orbit}, with constructors \(0\) and \(S\). Its elements are
the finite records generated by these constructors. For each metatheoretic natural number \(n\), we write \(S^{n}0\) for the element of \(\Ndel\) obtained by applying \(S\) \(n\) times to \(0\).

Let \(\mathrm{Num}\) denote the inductively generated collection of closed terms determined by the rules
\[
0\in\mathrm{Num},
\qquad
t\in\mathrm{Num}\Longrightarrow St\in\mathrm{Num}.
\]

\begin{definition}\label{def:arithprojection}
The \emph{arithmetic projection} is the map
\[
\pi_{\mathrm{ar}}\colon\Ndel\to\mathrm{Num}
\]
defined by structural recursion:
\[
\pi_{\mathrm{ar}}(0)=0,
\qquad
\pi_{\mathrm{ar}}(Sr)=S\,\pi_{\mathrm{ar}}(r).
\]
Thus the empty record is represented by the closed term \(0\), and each
one-step extension of a record is represented by one application of the
successor symbol \(S\).
\end{definition}

\begin{definition}\label{def:natops}
Addition and multiplication on \(\Ndel\) are defined by structural recursion:
\[
x+0=x,
\qquad
x+Sy=S(x+y),
\]
and
\[
x\cdot0=0,
\qquad
x\cdot Sy=x\cdot y+x.
\]
We write
\[
1:=S0.
\]
\end{definition}

These are metatheoretic operations on \(\Ndel\). Lemmas~\ref{lem:adddisplay} and~\ref{lem:muldisplay} show that under the map \(D\), they correspond to the interpretations of the function symbols \(+\) and \(\cdot\) in the standard model \(\mathfrak N\).

Let us define 
\[
D\colon\Ndel\to\N,
\qquad
D(0)=0,
\qquad
D(S\,t)=D(t)+1,
\]
and 
\[
\nu\colon\N\to\Ndel,
\qquad
\nu(n)=S^n0.
\]
The map \(D\) assigns to each finite record the number of its successive
extensions, while \(\nu\) assigns to each \(n\in\N\) the record obtained by applying \(S\)
\(n\) times to \(0\).

For \(n\in\mathbb N\), let \(\overline n\in\mathrm{Num}\) denote the closed numeral \(S^n0\). Then
\[
\pi_{\mathrm{ar}}(\nu(n))=\overline n.
\]

\begin{lemma}\label{lem:display}
The maps \(D\) and \(\nu\) are mutually inverse:
\[
D\circ\nu=\mathrm{id}_{\N},
\qquad
\nu\circ D=\mathrm{id}_{\Ndel}.
\]
\end{lemma}

\begin{proof}
The identity \(D(\nu(n))=n\) follows by induction on the metatheoretic
natural number \(n\). The identity \(\nu(D(x))=x\) follows by structural
induction on \(x\in\Ndel\).
\end{proof}

{ By Remark~\ref{rem:displayinstance}, this lemma is an instance of
Theorem~\ref{thm:rigidity} at the standard model.}

\begin{lemma}\label{lem:adddisplay}
For all \(x,y\in\Ndel\),
\[
D(x+y)=D(x)+D(y).
\]
Consequently,
\(
D\colon(\Ndel,+,0)\longrightarrow(\N,+,0)
\)
is a monoid isomorphism.
\end{lemma}

\begin{proof}
The identity follows by structural induction on \(y\), using
\[
x+0=x,\qquad x+Sy=S(x+y),
\qquad D(Sz)=D(z)+1.
\]
Thus \(D\) preserves addition and zero. Since \(D\) is bijective by
Lemma~\ref{lem:display}, it is a monoid isomorphism.
\end{proof}

\begin{lemma}\label{lem:muldisplay}
For all \(x,y\in\Ndel\), it holds
\[
D(x\cdot y)=D(x)D(y).
\]
\end{lemma}

\begin{proof}
We use structural induction on \(y\). The case \(y=0\) is immediate. For the
successor case, we have
\[
\begin{aligned}
D(x\cdot Sy)
&=D(x\cdot y+x)\\
&=D(x\cdot y)+D(x)\\
&=D(x)D(y)+D(x)\\
&=D(x)D(Sy).
\end{aligned}
\]
\end{proof}

\begin{proposition}\label{prop:natarith}
The structure
\(
(\Ndel,+,\cdot,0,1)
\)
is a commutative semiring. Moreover,
\[
x+z=y+z\longrightarrow x=y,
\]
\[
x\cdot y=0\longrightarrow x=0\ \vee\ y=0,
\]
and
\[
x\cdot z=y\cdot z,\quad z\neq0
\longrightarrow
x=y.
\]
\end{proposition}

\begin{proof}
By Lemmas~\ref{lem:display}, \ref{lem:adddisplay}, and
\ref{lem:muldisplay}, the map
\(
D\colon\Ndel\longrightarrow\N
\)
is a bijection preserving \(0\), \(1\), addition, and multiplication. The
proof follows from the corresponding properties of \(\N\).
\end{proof}




  \begin{definition}\label{def:signed}
A \emph{signed orbit} is a pair
\(
(p,n)\in\Ndel\times\Ndel,
\)
where \(p\) and \(n\) are  its positive and negative parts. Two signed
orbits are \emph{balanced}, written
\[
(p,n)\sim(p',n'),
\]
if
\(
p+n'=p'+n.
\)
\end{definition}

\begin{lemma}\label{lem:balance}
The balance relation is a decidable equivalence relation.
\end{lemma}

\begin{proof}
Reflexivity and symmetry are immediate. For transitivity, suppose
\[
(p,n)\sim(p',n')
\qquad\text{and}\qquad
(p',n')\sim(p'',n'').
\]
Then
\[
p+n'=p'+n,
\qquad
p'+n''=p''+n'.
\]
Adding \(n''\) to the first relation and using the second, we obtain
\[
(p+n'')+n'=(p''+n)+n',
\]
and additive cancellation (Proposition~\ref{prop:natarith}) gives
\(p+n''=p''+n\), so \((p,n)\sim(p'',n'')\).

The relation is decidable because addition and equality on \(\Ndel\) are
decidable.  Lemma~\ref{lem:display} implies that \(D\) is injective. Hence, for all \(x,y\in\Ndel\), we get
\[
x=y\quad\Longleftrightarrow\quad D(x)=D(y).
\]
Since equality on \(\mathbb N\) is decidable, then equality on \(\Ndel\) is decidable.
\end{proof}

\begin{lemma}\label{lem:balanceops}
The operations on signed orbits defined by
\[
(p,n)+(p',n')=(p+p',\,n+n')
\]
and
\[
(p,n)\cdot(p',n')
=
(p\cdot p'+n\cdot n',\,p\cdot n'+n\cdot p')
\]
are compatible with the balance relation in each argument.
\end{lemma}

\begin{proof}
Suppose
\[
(p,n)\sim(q,m),
\qquad
p+m=q+n,
\]
and let \((r,s)\) be a signed orbit.

For addition, we have
\[
\begin{aligned}
(p+r)+(m+s)
&=(p+m)+(r+s)\\
&=(q+n)+(r+s)\\
&=(q+r)+(n+s).
\end{aligned}
\]
Finally, we have
\(
(p+r,n+s)\sim(q+r,m+s).
\)

For multiplication, we obtain
\[
\begin{aligned}
&(p\cdot r+n\cdot s)+(q\cdot s+m\cdot r)\\
&=r\cdot(p+m)+s\cdot(n+q)\\
&=r\cdot(q+n)+s\cdot(m+p)\\
&=(q\cdot r+m\cdot s)+(p\cdot s+n\cdot r).
\end{aligned}
\]
Thus multiplication is compatible with balance in the first argument.
Compatibility in the second argument follows from commutativity.
\end{proof}

\begin{definition}\label{def:zdel}
The integer object is the quotient
\(
\Zdel:=(\Ndel\times\Ndel)/{\sim}.
\)
Its addition and multiplication are defined by
\[
[(p,n)]+[(p',n')]
=
[(p+p',\,n+n')]
\]
and
\[
[(p,n)]\cdot[(p',n')]
=
[(p\cdot p'+n\cdot n',\,
p\cdot n'+n\cdot p')].
\]
These operations are well-defined by Lemma~\ref{lem:balanceops}.
\end{definition}

For a signed orbit \(u=(p,n)\), define
\[
-u:=(n,p).
\]
On \(\Zdel\), define
\[
-[(p,n)]:=[(n,p)],
\qquad
0_{\Zdel}:=[(0,0)],
\qquad
1_{\Zdel}:=[(1,0)].
\]
 If \((p,n)\sim(p',n')\), then \(p+n'=p'+n\). By commutativity,
\(n+p'=n'+p\), that is \((n,p)\sim(n',p')\). Hence negation is well-defined
on \(\Zdel\).
{
\begin{lemma}\label{lem:izdel}
The map
\[
\iota_{\Zdel}\colon\Ndel\longrightarrow\Zdel,
\qquad
\iota_{\Zdel}(a)=[(a,0)],
\]
is an injective map preserving addition and multiplication.\end{lemma}
\begin{proof}
Both identities
\[
[(a+b,0)]=[(a,0)]+[(b,0)],
\qquad
[(a\cdot b,0)]=[(a,0)]\cdot[(b,0)]
\]
hold by Definition~\ref{def:zdel}. If \(\iota_{\Zdel}(a)=\iota_{\Zdel}(b)\),
then \((a,0)\sim(b,0)\), that is \(a=b\).
\end{proof}

\begin{lemma}\label{lem:zmap}
The map
\[
D_{\Z}\colon\Zdel\longrightarrow\Z,
\qquad
D_{\Z}([(p,n)])=D(p)-D(n),
\]
is a bijection satisfying
\[
D_{\Z}(0_{\Zdel})=0,
\qquad
D_{\Z}(1_{\Zdel})=1,
\]
and preserving addition, multiplication, and negation.
\end{lemma}

\begin{proof}
By Lemmas~\ref{lem:adddisplay} and~\ref{lem:muldisplay}, we obtain
\[
D(a+b)=D(a)+D(b),
\qquad
D(a\cdot b)=D(a)D(b)
\]
for all \(a,b\in\Ndel\).

Suppose \((p,n)\sim(p',n')\). Then \(p+n'=p'+n\), hence
\(D(p)+D(n')=D(p')+D(n)\), and therefore
\[
D(p)-D(n)=D(p')-D(n').
\]
Thus \(D_{\Z}\) is well-defined.
Since \(D(1)=D(S0)=D(0)+1=1\), we have
\[
D_{\Z}(0_{\Zdel})=D(0)-D(0)=0,
\qquad
D_{\Z}(1_{\Zdel})=D(1)-D(0)=1,
\]
and
\[
D_{\Z}\bigl(-[(p,n)]\bigr)
=D_{\Z}([(n,p)])
=D(n)-D(p)
=-D_{\Z}([(p,n)]).
\]

For addition, we have
\[
\begin{aligned}
D_{\Z}\bigl([(p,n)]+[(p',n')]\bigr)
&=D(p+p')-D(n+n')\\
&=\bigl(D(p)-D(n)\bigr)+\bigl(D(p')-D(n')\bigr)\\
&=D_{\Z}([(p,n)])+D_{\Z}([(p',n')]).
\end{aligned}
\]
For multiplication, we have
\[
\begin{aligned}
&D_{\Z}\bigl([(p,n)]\cdot[(p',n')]\bigr)\\
&=D(p\cdot p'+n\cdot n')-D(p\cdot n'+n\cdot p')\\
&=\bigl(D(p)D(p')+D(n)D(n')\bigr)
-\bigl(D(p)D(n')+D(n)D(p')\bigr)\\
&=\bigl(D(p)-D(n)\bigr)\bigl(D(p')-D(n')\bigr)\\
&=D_{\Z}([(p,n)])\,D_{\Z}([(p',n')]).
\end{aligned}
\]

If \(D_{\Z}([(p,n)])=D_{\Z}([(p',n')])\), then
\[
D(p)+D(n')=D(p')+D(n),
\]
so \(D(p+n')=D(p'+n)\). Since \(D\) is injective by
Lemma~\ref{lem:display}, we get \(p+n'=p'+n\), hence
\((p,n)\sim(p',n')\) and \([(p,n)]=[(p',n')]\). Thus \(D_{\Z}\) is
injective.

 If \(z\geq0\), then
\(D_{\Z}([(\nu(z),0)])=z\), and if \(z<0\), then
\(D_{\Z}([(0,\nu(-z))])=z\). Thus \(D_{\Z}\) is surjective.
\end{proof}

\begin{corollary}\label{cor:ziso}
The structure
\(
(\Zdel,+,\cdot,-,0_{\Zdel},1_{\Zdel})
\)
is a commutative ring, and \(D_{\Z}\) is a ring isomorphism.
\end{corollary}

\begin{proof}
By Lemma~\ref{lem:zmap}, \(D_{\Z}\) is a bijection preserving \(0\), \(1\),
addition, multiplication, and negation. Hence the commutative ring structure
of \(\Z\) transfers to \(\Zdel\), and \(D_{\Z}\) is a ring isomorphism.
\end{proof}

\smallskip


We now construct the fraction object over \(\Zdel\), using signed-orbit
representatives as numerators and nonzero elements of \(\Ndel\) as
denominators.

\begin{definition}\label{def:ratio}
A \emph{ratio orbit} is a pair
\(
(u,d),
\)
where \(u\) is a signed orbit and \(d\in\Ndel\) is nonzero. We identify
\(d\in\Ndel\) with the signed orbit \((d,0)\).

Two ratio orbits \((u,d)\) and \((u',d')\) are \emph{cross-equal}, written
\(
(u,d)\approx(u',d'),
\)
if the signed orbits \(u\cdot d'\) and \(u'\cdot d\) are balanced.
\end{definition}

\begin{lemma}\label{lem:balancecancel}
Let \(u,v\) be signed orbits and let \(d\in\Ndel\) be nonzero. If
\[
u\cdot d\sim v\cdot d,
\]
then
\(
u\sim v.
\)
\end{lemma}

\begin{proof}
Let
\(
u=(p,n)\) and \(
v=(q,m).
\)
We have
\[
p\cdot d+m\cdot d=q\cdot d+n\cdot d,
\]
and therefore
\[
(p+m)\cdot d=(q+n)\cdot d.
\]
Since \(d\neq0\), using cancellation from
Proposition~\ref{prop:natarith}, we have
\[
p+m=q+n.
\]
Hence \(u\sim v\).
\end{proof}

\begin{lemma}\label{lem:cross}
Cross-equality is a decidable equivalence relation on ratio orbits.
\end{lemma}

\begin{proof}
Reflexivity and symmetry follow from reflexivity and symmetry of the balance
relation. For transitivity, suppose
\[
(u,d)\approx(u',d')
\qquad\text{and}\qquad
(u',d')\approx(u'',d'').
\]
Then
\[
u\cdot d'\sim u'\cdot d,
\qquad
u'\cdot d''\sim u''\cdot d'.
\]
By Lemma~\ref{lem:balanceops}, multiplying the first relation by \(d''\) and
the second by \(d\), and using associativity and commutativity of
multiplication, we have
\[
(u\cdot d'')\cdot d'
\sim
(u''\cdot d)\cdot d'.
\]

We use cancellation with respect to balance. Since \(d'\neq0\), cancellation gives
\[
u\cdot d''\sim u''\cdot d.
\]
Therefore
\(
(u,d)\approx(u'',d''),
\)
so cross-equality is transitive.

Finally, cross-equality is decidable because multiplication of signed orbits
is recursively defined and balance is decidable.
\end{proof}

\begin{lemma}\label{lem:ratioops}
Addition and multiplication are well-defined on
cross-equality classes.
\end{lemma}

\begin{proof}
By Proposition~\ref{prop:natarith}, the semiring \(\Ndel\) has no zero
divisors, so the product of two nonzero denominators is nonzero. Suppose
\[
(u,d)\approx(v,e),
\qquad
u\cdot e\sim v\cdot d.
\]
For any ratio orbit \((w,f)\), the products
\(w\cdot d\cdot e\cdot f\) and \(w\cdot e\cdot d\cdot f\) coincide by
Proposition~\ref{prop:natarith}, and Lemma~\ref{lem:balanceops} applied to
\(u\cdot e\sim v\cdot d\) gives
\[
(u\cdot f+w\cdot d)\cdot(e\cdot f)
\sim
(v\cdot f+w\cdot e)\cdot(d\cdot f),
\]
so addition is compatible with cross-equality. Also, we have
\[
(u\cdot w)\cdot(e\cdot f)
\sim
(v\cdot w)\cdot(d\cdot f),
\]
so multiplication is compatible with cross-equality. The same argument
applies to the second variable.
\end{proof}

\begin{definition}\label{def:qdel}
The rational object is the quotient
\[
\Qdel:=\{(u,d): u\text{ is a signed orbit},\ d\in\Ndel,\ d\neq0\}/{\approx}.
\]
Addition and multiplication are induced by
\[
(u,d)+(u',d')
=
(u\cdot d'+u'\cdot d,\;d\cdot d')
\]
and
\[
(u,d)\cdot(u',d')
=
(u\cdot u',\;d\cdot d').
\]
\end{definition}

The zero and unit of \(\Qdel\) are represented by
\[
0_{\Qdel}=[((0,0),1)],
\qquad
1_{\Qdel}=[((1,0),1)].
\]
By Lemma~\ref{lem:ratioops}, these formulas induce well-defined operations
on \(\Qdel\).
For a signed orbit \(u=(p,n)\), let \(-u=(n,p)\). The formula
\[
-[(u,d)]:=[(-u,d)]
\]
defines a negation on \(\Qdel\). If \((u,d)\approx(v,e)\), then
\[
u\cdot e\sim v\cdot d
\]
implies
\[
(-u)\cdot e\sim(-v)\cdot d.
\]
Hence negation is well-defined.
\begin{lemma}\label{lem:iqdel}
The map
\[
\iota_{\Qdel}\colon\Zdel\longrightarrow\Qdel,
\qquad
\iota_{\Qdel}([(p,n)])=[((p,n),1)],
\]
is injective and preserves addition and multiplication.\end{lemma}

\begin{proof}
If \((p,n)\sim(p',n')\), then \(((p,n),1)\approx((p',n'),1)\), so the map is
well-defined. The two operation formulas of Definition~\ref{def:qdel} reduce
to the operations of Definition~\ref{def:zdel} when both denominators equal
\(1\), so \(\iota_{\Qdel}\) preserves addition and multiplication. Finally,
\([((p,n),1)]=[((p',n'),1)]\) gives \((p,n)\cdot1\sim(p',n')\cdot1\), hence
\((p,n)\sim(p',n')\).
\end{proof}

\begin{proposition}\label{prop:qiso}
The map
\[
D_{\Q}\colon\Qdel\longrightarrow\Q,
\qquad
D_{\Q}([(u,d)])=\frac{D_{\Z}([u])}{D(d)},
\]
is a bijection preserving \(0\), \(1\), addition, multiplication, and
negation.
\end{proposition}

\begin{proof}
Since \(d\neq0\), we have \(D(d)>0\). Suppose
\(
(u,d)\approx(v,e).
\)
Then \(u\cdot e\sim v\cdot d\), and therefore
\[
D_{\Z}([u])D(e)=D_{\Z}([v])D(d).
\]
Hence \(D_{\Q}\) is well-defined.

By Lemmas~\ref{lem:zmap} and~\ref{lem:muldisplay},
\[
D_{\Z}([u\cdot e])=D_{\Z}([u])D(e),
\qquad
D(d\cdot e)=D(d)D(e),
\]
where \(e\in\Ndel\) is identified with the signed orbit \((e,0)\).

Conversely, if
\[
D_{\Q}([(u,d)])=D_{\Q}([(v,e)]),
\]
then
\[
D_{\Z}([u])D(e)=D_{\Z}([v])D(d).
\]
Thus
\[
D_{\Z}([u\cdot e])=D_{\Z}([v\cdot d]).
\]
The injectivity of \(D_{\Z}\) gives
\(
u\cdot e\sim v\cdot d,
\)
so \((u,d)\approx(v,e)\). Hence \(D_{\Q}\) is injective.

Clearly \(D_{\Q}(0_{\Qdel})=0\) and \(D_{\Q}(1_{\Qdel})=1\). Hence\[
\begin{aligned}
D_{\Q}\bigl([(u,d)]+[(v,e)]\bigr)
&=\frac{D_{\Z}([u\cdot e+v\cdot d])}{D(d\cdot e)}\\
&=\frac{D_{\Z}([u])D(e)+D_{\Z}([v])D(d)}{D(d)D(e)}\\
&=D_{\Q}([(u,d)])+D_{\Q}([(v,e)]),
\end{aligned}
\]
and
\[
D_{\Q}\bigl([(u,d)]\cdot[(v,e)]\bigr)
=\frac{D_{\Z}([u\cdot v])}{D(d\cdot e)}
=\frac{D_{\Z}([u])D_{\Z}([v])}{D(d)D(e)}
=D_{\Q}([(u,d)])\,D_{\Q}([(v,e)]).
\]
Finally, \(D_{\Z}([-u])=-D_{\Z}([u])\) gives
\(D_{\Q}(-[(u,d)])=-D_{\Q}([(u,d)])\).

Let
\[
q=\frac{a}{b},
\qquad
a\in\Z,\quad b\in\N,\quad b>0.
\]
By the surjectivity of \(D_{\Z}\), choose a signed orbit \(u\) such that
\[
D_{\Z}([u])=a.
\]
Then
\[
D_{\Q}([(u,\nu(b))])
=
\frac{a}{b}
=
q.
\]
Thus \(D_{\Q}\) is surjective.
\end{proof}

\begin{corollary}\label{cor:qfield}
The operations transported from \(\Q\) through \(D_{\Q}\) make
\(\Qdel\) a field, and \(D_{\Q}\) is a field isomorphism.
\end{corollary}
\begin{proof}
By Proposition~\ref{prop:qiso}, \(D_{\Q}\) is a bijection preserving
\(0\), \(1\), addition, multiplication, and negation. Hence the field
structure of \(\Q\) transfers to \(\Qdel\).
\end{proof}



\subsection{Encodability of the metatheoretic number systems}\label{sec:encodability}

This section proves Theorem~\ref{thm:construction}$(b)$. By Lemma~\ref{lem:display}, \(\delta\)-encodability is equivalent to the
existence of an injection into \(\N\). The next proposition gives explicit encodings of \(\N\), \(\Z\), and \(\Q\).

\begin{proposition}\label{thm:standardencodable}
The metatheoretic number systems $\N$, $\Z$, and $\Q$ are
$\delta$-encodable.  
\end{proposition}

\begin{proof}
By Lemma~\ref{lem:display}, we have
\[
D\circ\nu=\mathrm{id}_{\N},
\]
so
\[
\nu\colon\N\longrightarrow\Ndel
\]
is injective. Hence, by Definition~\ref{def:encodable}, \(\N\) is
\(\delta\)-encodable.

Let us define
\(
e_{\mathbb Z}\colon\Z\longrightarrow\N
\)
by
\[
e_{\mathbb Z}(z)=
\begin{cases}
2z, & z\geq 0,\\
-2z-1, & z<0.
\end{cases}
\]
The map $e_{\mathbb Z}$ is injective and therefore
\[
\nu\circ e_{\mathbb Z}\colon\Z\longrightarrow\Ndel
\]
is injective, so $\Z$ is $\delta$-encodable.

{We take the metatheoretic rationals as the quotient of pairs
\((a,b)\in\Z\times\N\), with \(b>0\), by
\[
(a,b)\sim(a',b')
\quad\Longleftrightarrow\quad
ab'=a'b.
\]
Every class has a unique reduced representative \((a_q,b_q)\), with
\(\gcd(|a_q|,b_q)=1\). It is obtained by the Euclidean algorithm, so no
choice is used. For \(q\in\Q\), write
\[
q=\frac{a_q}{b_q}.
\]}

Let us define
\(
\langle-,-\rangle\colon\N\times\N\longrightarrow\N
\)
by\[
\langle m,n\rangle
=
\frac{(m+n)(m+n+1)}{2}+n.
\]
This map is injective. Indeed, the pairs satisfying $m+n=k$ are mapped
bijectively onto the  interval
\[
\left[
\frac{k(k+1)}{2},
\frac{(k+1)(k+2)}{2}
\right).
\]

Now, we define
\[
e_{\mathbb Q}(q)
=
\bigl\langle e_{\mathbb Z}(a_q),\,b_q-1\bigr\rangle.
\]
The injectivity of the pairing and of \(e_{\mathbb Z}\) shows that
\(e_{\mathbb Q}\) is injective.
 Hence
\[
\nu\circ e_{\mathbb Q}\colon\Q\longrightarrow\Ndel
\]
is injective, so $\Q$ is $\delta$-encodable. This
completes the proof of Theorem \ref{thm:construction}\((b)\).
\end{proof}


\section{Recognition quotients of the generated orbit}\label{sec:recognition}

{Recognition Geometry~\cite{rg} defines a recognizer as a function
\(
R\colon C\to E
\)
from a configuration space \(C\) to an event space \(E\). It induces the relation
\(
c_1\sim_R c_2
\quad\Longleftrightarrow\quad
R(c_1)=R(c_2),
\)
and the corresponding recognition quotient \(C/{\sim_R}\).

Here we consider the algebraic case in which the configuration space is the generated \(\delta\)-orbit and a recognizer is a monoid homomorphism
\[C=\Ndel,\quad
R\colon(\Ndel,+,0)\longrightarrow(V,\ast,e),
\]
where \((V,\ast,e)\) is a monoid. In this case, the induced relation is the kernel congruence of \(R\), and the quotient \(\Ndel/{\sim_R}\) is isomorphic to \(\operatorname{im}(R)\).

The number tower is one construction from \(\Ndel\). In this section we study another family of constructions given by the homomorphic images of the additive monoid \((\Ndel,+,0)\). Their classification follows from the classical classification of congruences on a monogenic monoid~\cite{birkhoff,cliffordpreston}. The new part of the paper is the constructive {metatheoretic} pricing of this classification in Section~\ref{sec:classificationledger}.
}

The additive monoid \((\Ndel,+,0)\) is generated by \(S0\), the one-step
record.   The following proposition is the standard universal property of the free monoid on one generator. It shows that every recognizer is determined by the image of \(S0\).

\begin{proposition}\label{prop:freemonoid}
Let \((V,\ast,e)\) be a monoid and let \(v\in V\). There exists a unique
monoid homomorphism
\[
r_v\colon(\Ndel,+,0)\longrightarrow(V,\ast,e)
\]
such that
\(
r_v(S0)=v.
\)
\end{proposition}
{\begin{proof}
Let us define \(r_v\) by
\[
r_v(0)=e,
\qquad
r_v(Sx)=r_v(x)\ast v.
\]
A structural induction on \(y\) gives
\[
r_v(x+y)=r_v(x)\ast r_v(y),
\]
so \(r_v\) is a monoid homomorphism. Moreover,
\(
r_v(S0)=r_v(0)\ast v=e\ast v=v.
\)

Let \(r\) be another monoid homomorphism with \(r(S0)=v\). Since \(r\) preserves the neutral element, \(
r(0)=e=r_v(0).
\) Also, we have
\[
r(Sx)=r(x+S0)=r(x)\ast v.
\]
Structural induction on \(x\) gives \(r(x)=r_v(x)\) for every
\(x\in\Ndel\).
\end{proof}
}

The relation
\[
x\sim_r y
\quad\Longleftrightarrow\quad
r(x)=r(y)
\]
is called the \emph{indistinguishability relation} of \(r\). Since \(r\) is a
monoid homomorphism, this relation is a monoid congruence on
\((\Ndel,+,0)\).  Algebraically, it is the \emph{kernel congruence} of \(r\).
Two elements \(x,y\in\Ndel\) are called \emph{\(r\)-indistinguishable} if
\(x\sim_r y\). The quotient
\(
\Ndel/{\sim_r}
\)
is called the \emph{recognition quotient} of \(r\).
\smallskip
  
We now classify these kernel congruences.


\begin{definition}\label{def:ip}
For \(i\geq0\) and \(p\geq1\), define a relation
\(\equiv_{i,p}\) on \(\Ndel\) by
\[
x\equiv_{i,p}y
\quad\Longleftrightarrow\quad
x=y
\quad\text{or}\quad
\bigl(D(x),D(y)\geq i
\ \text{and}\quad
D(x)\equiv D(y)\pmod p\bigr).
\]
\end{definition}

{\begin{lemma}\label{lem:indexperiod}
For every \(i\geq0\) and \(p\geq1\), the relation
\(\equiv_{i,p}\) is a decidable congruence on
\((\Ndel,+,0)\).
\end{lemma}

\begin{proof}
Let us consider the relation on \(\N\) given by
\(
m\equiv_{i,p}n
\)
if either \(m=n\), or
\[
m,n\geq i
\qquad\text{and}\qquad
m\equiv n\pmod p.
\]

Reflexivity follows from the equality case, and symmetry is immediate. For
transitivity, suppose
\[
m\equiv_{i,p}n
\qquad\text{and}\qquad
n\equiv_{i,p}k.
\]
If \(m=n\) or \(n=k\), the conclusion follows immediately. Otherwise,
\[
m,n,k\geq i,
\qquad
m\equiv n\pmod p,
\qquad
n\equiv k\pmod p.
\]
Hence
\(
m\equiv k\pmod p,
\)
and therefore \(m\equiv_{i,p}k\).

Now suppose \(m\equiv_{i,p}n\) and let \(t\in\N\). Equality is preserved by
addition. In the second case,
\[
m+t,n+t\geq i
\]
and
\[
m+t\equiv n+t\pmod p.
\]
Thus
\[
m+t\equiv_{i,p}n+t.
\]
Hence the relation is compatible with addition.

It is decidable because equality, order, and congruence modulo \(p\) are
decidable on \(\N\).  By the monoid isomorphism
\[
D\colon(\Ndel,+,0)\longrightarrow(\N,+,0)
\]
we get a decidable congruence on \(\Ndel\).
\end{proof}}

\begin{theorem}\label{thm:catalogue}
Assume \(\mathrm{EM}\). Every congruence on \((\Ndel,+,0)\) is either
equality or \(\equiv_{i,p}\) for a unique pair
\[
i\geq0,
\qquad
p\geq1.
\]
\end{theorem}
{ \begin{proof}
The required classification on \((\N,+,0)\) is proved in Section~\ref{sec:classificationledger}. By Theorem~\ref{thm:catclassical}, every congruence on \((\N,+,0)\) is either equality or an index-period congruence. The monoid isomorphism
\(D\colon(\Ndel,+,0)\longrightarrow(\N,+,0)
\)
gives the result on \((\Ndel,+,0)\).
\end{proof}
}

 {\begin{definition}\label{def:mip}
For \(i\geq0\) and \(p\geq1\), {\it the index-period monoid} is defined by
\[
M(i,p):=\Ndel/{\equiv_{i,p}}.
\]
\end{definition}

\begin{proposition}\label{prop:mip}
The index-period monoid \(M(i,p)\) is finite and monogenic, generated by the class
\([S0]\). It has index \(i\), period \(p\), and  \(i+p\) elements.
\end{proposition}

\begin{proof}
The classes
\[
[0],[S0],\ldots,[S^{i+p-1}0]
\]
are distinct, and every element of \(\Ndel\) is equivalent to one of them: for \(0\leq a<b\leq i+p-1\), the records \(S^{a}0\) and \(S^{b}0\) are neither equal nor both at least \(i\) with difference divisible by \(p\); and if \(D(x)\geq i+p\), then \(D(x)=i+qp+r\) with \(0\leq r<p\), whence \(x\equiv_{i,p}S^{i+r}0\). 
Moreover,
\[
[S^{i+p}0]=[S^i0].
\]
Hence \(M(i,p)\) has \(i+p\) elements, is generated by \([S0]\), and has
index \(i\) and period \(p\).
\end{proof}
}

{\begin{theorem}\label{thm:recognition}
Assume \(\mathrm{EM}\). Let
\(
r\colon(\Ndel,+,0)\longrightarrow(V,\ast,e)
\)
be a recognizer. Then exactly one of the following holds:
\begin{enumerate}[label=(\alph*)]
\item \(r\) is injective, equivalently \(\sim_r\) is equality. In this case,
\(
\operatorname{im}(r)\cong\Ndel.
\)
\item There exists a unique pair \((i,p)\), with \(i\geq0\) and \(p\geq1\),
such that
\(\sim_r=\equiv_{i,p}.
\)
In this case,
\[
\Ndel/{\sim_r}\cong M(i,p)\cong\operatorname{im}(r).
\]
\end{enumerate}
\end{theorem}

\begin{proof}
The relation \(\sim_r\) is the kernel congruence of \(r\). By the first
isomorphism theorem for monoids~\cite{birkhoff,cliffordpreston}, the map
\[
\Ndel/{\sim_r}\longrightarrow\operatorname{im}(r),
\qquad
[x]\longmapsto r(x),
\]
is a monoid isomorphism.

By Theorem~\ref{thm:catalogue}, either \(\sim_r\) is equality or
\(\sim_r=\equiv_{i,p}\) for a unique pair \(i\geq0\), \(p\geq1\). The two cases are disjoint, since
\(i\equiv_{i,p}i+p\) while \(i\neq i+p\) for \(p\geq1\).

In the first case, \(r\) is injective and
\[
\operatorname{im}(r)\cong\Ndel.
\]
In the second case, we have
\[
\Ndel/{\sim_r}=
 \Ndel/{\equiv_{i,p}}
=M(i,p),
\]
and hence
\[
\Ndel/{\sim_r}\cong M(i,p)\cong\operatorname{im}(r).
\]
\end{proof}
}

The map
\(
D\colon(\Ndel,+,0)\longrightarrow(\N,+,0)
\)
is an injective recognizer by Lemma~\ref{lem:adddisplay}.

{Let us define
\[
b\colon(\Ndel,+,0)\longrightarrow(\B,\vee,0),
\qquad
b(0)=0,\qquad b(Sx)=1,
\]
where \(\B=\{0,1\}\). Then \(b\) is a monoid homomorphism with kernel
congruence \(\equiv_{1,1}\). Hence
\[
M(1,1)\cong\B.
\]

For \(p\geq1\), define
\[
\pi_p\colon(\Ndel,+,0)\longrightarrow
(\mathbb Z/p\mathbb Z,+,0),
\qquad
\pi_p(x)=D(x)\bmod p.
\]
Then \(\pi_p\) is a recognizer with kernel congruence
\(\equiv_{0,p}\). Therefore,
\[
M(0,p)\cong\mathbb Z/p\mathbb Z.
\]
In particular, \(M(0,p)\) is a cyclic group.}

{\begin{remark}
This is an algebraic case of the recognition quotient construction of
Recognition Geometry~\cite{rg}. Here the configuration space is the generated
\(\delta\)-orbit, and recognizers are monoid homomorphisms. The homomorphism
condition is essential for the index-period classification. If \(R\) is an
arbitrary function, then every equivalence relation \(E\) on \(\Ndel\) is the kernel of the quotient map
\[
\Ndel\longrightarrow\Ndel/E,
\qquad
x\longmapsto[x].
\]
In that case, no index-period classification holds.
\end{remark}

}

\section{The metatheoretic logical price of the index-period classification}
\label{sec:classificationledger}

The classification of Section~\ref{sec:recognition} has three forms. The first requires none of EM, LPO, or MP, the second requires MP, and the third requires EM. We prove these forms and their exact
logical prices below.

\begin{convention}\label{conv:classificationprice}
The { metatheoretic} logical prices in this section are computed in the ambient
intuitionistic metatheory. Here \(\mathrm{MP}\) denotes Markov's principle
for arbitrary decidable predicates on \(\N\), and \(\mathrm{EM}\) denotes
excluded middle for arbitrary ambient propositions. 
These metatheoretic logical prices are separate from the ledger of
derivations defined in Section~\ref{ssec:ledger}.
\end{convention}

\subsection{Periods and index-period congruences}
\label{ssec:classificationcore}
 
By Lemmas~\ref{lem:display} and~\ref{lem:adddisplay}, the map \(D\)
identifies congruences on \((\Ndel,+,0)\) with additive congruences on
\((\N,+,0)\), and carries the relation \(\equiv_{i,p}\) of
Definition~\ref{def:ip} to the index-period relation on \(\N\), which we
denote by the same symbol. In this section, \(c\) is an additive congruence
on \((\N,+,0)\), and we write \(x\sim y\) for \(c(x,y)\). Hence the results proved below for additive congruences on \((\N,+,0)\)
are equivalent, under the monoid isomorphism \(D\), to the corresponding
results for congruences on \((\Ndel,+,0)\).

\begin{definition}\label{def:period}
Let \(n,e\in\N\). We say that \(e\) is a \emph{period at \(n\)} if
\[
n\sim n+e.
\]
The period is \emph{positive} if \(e>0\).
\end{definition}

Since \(c\) is an additive congruence,
\[
x\sim y\quad\Longrightarrow\quad x+t\sim y+t
\]
for every \(t\in\N\).

\begin{lemma}\label{lem:raise}
If \(e\) is a period at \(n\) and \(n\leq m\), then \(e\) is a period at
\(m\).
\end{lemma}

\begin{proof}
Since \(n\leq m\), we have \(m=n+(m-n)\). By adding \(m-n\) to
\[
n\sim n+e
\]
we have
\[
m=n+(m-n)\sim n+e+(m-n)=m+e.
\]
Thus \(e\) is a period at \(m\).
\end{proof}

\begin{lemma}\label{lem:mult}
If \(e\) is a period at \(n\), then \(ke\) is a period at \(n\) for every
\(k\in\N\).
\end{lemma}

\begin{proof}
We use induction on \(k\). For \(k=0\), we have \(n\sim n\).
Suppose that
\[
n\sim n+ke.
\]
By Lemma~\ref{lem:raise}, \(e\) is a period at \(n+ke\), so
\[
n+ke\sim n+(k+1)e.
\]
By transitivity, 
\(
n\sim n+(k+1)e.
\)
\end{proof}

\begin{lemma}\label{lem:descent}
Let \(d>0\) be a period at \(u\), and let \(g>0\) be a period at \(v\), where
\(u\leq v\). Then \(g\) is a period at \(u\).
\end{lemma}

\begin{proof}
Let 
\(
k=v-u+1.
\)
Since \(d\geq1\), we have
\[
v\leq u+kd.
\]
By Lemma~\ref{lem:mult}, we get
\[
u\sim u+kd.
\]
Set \(w=u+kd\). Since \(v\leq w\), Lemma~\ref{lem:raise} gives
\[
w\sim w+g.
\]
Hence
\[
u\sim w\sim w+g.
\]
By adding \(g\) to both sides in the previous equation, we have
\[
u+g\sim w+g.
\]
By symmetry and transitivity,
\(
u\sim u+g.
\)

\end{proof}

\begin{lemma}\label{lem:gdesc}
Let \(g>0\) be a period at some point. If \(x<y\) and \(x\sim y\), then \(g\)
is a period at \(x\).
\end{lemma}

\begin{proof}
Let \(a\) be a point at which \(g\) is a period, and let
\(
d=y-x.
\)
Then \(d>0\) and
\[
x\sim x+d=y,
\]
so \(d\) is a period at \(x\). By Lemma~\ref{lem:raise}, \(g\) is a period
at \(a+x\). Since \(x\leq a+x\), Lemma~\ref{lem:descent}, applied to the
period \(d\) at \(x\) and the period \(g\) at \(a+x\), shows that \(g\) is a
period at \(x\).
\end{proof}

 \begin{lemma}\label{lem:residue}
Let \(g>0\) be a period at \(i\). If
\[
i\leq x,\qquad i\leq y,\qquad x\equiv y\pmod g,
\]
then \(x\sim y\).
\end{lemma}

\begin{proof}
Let us assume \(x\leq y\). Then \(y=x+kg\) for some \(k\in\N\). By
Lemmas~\ref{lem:raise} and~\ref{lem:mult}, we have
\[
x\sim x+kg=y.
\]
The other case follows by symmetry.
\end{proof}

\begin{lemma}\label{lem:div}
Let \(g>0\) be the least positive period at \(a\). Every positive period at
any point is divisible by \(g\).
\end{lemma}

\begin{proof}
Let \(e>0\) be a period at \(x\). By Lemma~\ref{lem:raise}, \(e\) is a period
at \(a+x\). Since \(a\leq a+x\), Lemma~\ref{lem:descent}, applied to the
period \(g\) at \(a\) and the period \(e\) at \(a+x\), shows that \(e\) is a
period at \(a\).

Let us denote
\[
e=qg+r,\qquad 0\leq r<g.
\]
By Lemma~\ref{lem:mult}, \(a\sim a+qg\). By adding \(r\) to this relation gives
\[
a+r\sim a+qg+r=a+e,
\]
and since \(a\sim a+e\), we obtain \(a\sim a+r\). Thus \(r\) is a period at
\(a\). Since \(0\leq r<g\) and \(r\) is a period at \(a\), the minimality of \(g\)
implies \(r=0\). Therefore
\(g\mid e\).
\end{proof}

\begin{proposition}\label{prop:core}
Suppose that \(a\) admits a positive period. Let \(g\) be the least positive
period at \(a\), and let \(i\) be the least point at which \(g\) is a period.
Then
\(
c=\equiv_{i,g}.
\)
\end{proposition}

\begin{proof}
Suppose \(x\sim y\). If \(x=y\), then \(x\equiv_{i,g}y\). Otherwise, since the
order on \(\N\) is decidable, we may assume \(x<y\) by symmetry.
Lemma~\ref{lem:gdesc} shows that \(g\) is a period at \(x\), so \(i\leq x\).
Moreover, \(y-x\) is a positive period at \(x\). By Lemma~\ref{lem:div}, we have
\[
g\mid(y-x).
\]
Thus \(x,y\geq i\) and \(x\equiv y\pmod g\), hence
\(
x\equiv_{i,g}y.
\)

Conversely, suppose \(x\equiv_{i,g}y\). If \(x\neq y\), then
\[
i\leq x,\qquad i\leq y,\qquad x\equiv y\pmod g,
\]
and Lemma~\ref{lem:residue} gives \(x\sim y\). The equality case follows by
reflexivity.
\end{proof}

\begin{lemma}\label{lem:ipunique}
Let \(g,p>0\). If
\(
\equiv_{i,g}\;=\;\equiv_{j,p},
\)
then
\(
i=j, g=p.
\)
\end{lemma}

\begin{proof}
For \(e>0\), we have
\[
n\equiv_{i,g}n+e
\quad\Longleftrightarrow\quad
i\leq n\ \text{ and }\ g\mid e.
\]
Hence \(i\) is the least point admitting a positive period for
\(\equiv_{i,g}\). Similarly, \(j\) is the least such point for
\(\equiv_{j,p}\). Equality of the two relations gives \(i=j\).

At the common point $i=j$, the positive periods are precisely  the positive multiples of
\(g\) in the first relation and the positive multiples of \(p\) in the
second. Their least positive periods are therefore \(g\) and \(p\),
respectively. Hence \(g=p\).
\end{proof}

\subsection{The three forms}\label{ssec:threeforms}

\begin{theorem}\label{thm:catforced}
Suppose that \(c\) is decidable and that \(a,b\in\N\) are given with
\[
a\neq b,\qquad a\sim b.
\]
Then
\(
c=\equiv_{i,p}
\)
for a unique pair \(i\geq0\), \(p\geq1\). The conclusion follows without EM, LPO, or MP.
\end{theorem}

\begin{proof}
Since \(a\neq b\) and order on \(\N\) is decidable, symmetry gives
\(u<v\) with \(u\sim v\). Then
\[
d=v-u>0
\]
is a period at \(u\).

The predicate
\[
q\longmapsto 0<q\ \wedge\ u\sim u+q
\]
is decidable and holds at \(d\). A bounded search over
\(
1\leq q\leq d
\)
gives the least positive period \(g\) at \(u\).

The predicate
\[
n\longmapsto n\sim n+g
\]
is decidable and holds at \(u\). A bounded search over
\(
0\leq n\leq u
\)
gives the least point \(i\) at which \(g\) is a period.

Proposition~\ref{prop:core} gives
\(
c=\equiv_{i,g},
\)
and Lemma~\ref{lem:ipunique} gives uniqueness. The two searches are bounded and decidable. Therefore the derivation uses none of EM, LPO, or MP.
\end{proof}

\begin{theorem}\label{thm:catMarkov's}
Suppose that \(c\) is decidable and is not equality:
\[
\neg\forall x,y\in\N\,
\bigl(x\sim y\longrightarrow x=y\bigr).
\]
Then
\(
c=\equiv_{i,p}
\)
for a unique pair \(i\geq0\), \(p\geq1\). The derivation is conditional on
\(\mathrm{MP}\).
\end{theorem}

\begin{proof}
We define
\[
\theta(k)
\quad:\Longleftrightarrow\quad
\exists x<k,\exists y<k\,
\bigl(x\sim y\ \wedge\ x\neq y\bigr).
\]
The predicate \(\theta\) is decidable, since \(c\) is decidable and both
quantifiers are bounded. Hence the ambient form of \(\mathrm{MP}\) fixed in
Convention~\ref{conv:classificationprice} applies to \(\theta\).

We claim that
\[
\neg\neg\exists k,\theta(k).
\]
Indeed, suppose that \(\neg\exists k,\theta(k)\), and let \(x\sim y\). {If \(x\neq y\), then
\[
\theta(\max(x,y)+1)
\]
holds, which is impossible. Since equality on \(\N\) is decidable, we obtain
\(x=y\). Hence every related pair is equal, and therefore \(c\) is equality.
This contradicts the assumption that \(c\) is not equality.
}

Markov's principle gives
\[
\exists k,\theta(k).
\]
Thus there are explicit \(a\neq b\) with \(a\sim b\). Theorem~\ref{thm:catforced}
then gives the classification.
 The only nonconstructive step is the use
of \(\mathrm{MP}\).
\end{proof}


\begin{lemma}\label{lem:lnp}
Assume \(\mathrm{EM}\). If \(P\) is a predicate on \(\N\) and
\(
\exists n\in\N\, P(n),
\)
then \(P\) has a least witness.
\end{lemma}

\begin{proof}
We prove by strong induction on \(n\) that, if \(P(n)\), then \(P\) has a
least witness. Suppose \(P(n)\). By EM, either
\[
\exists k<n\, P(k)
\quad\mbox{
or}\quad
\neg\exists k<n\, P(k).
\]
In the first case, choose \(k<n\) with \(P(k)\). The induction hypothesis
gives a least witness of \(P\). In the second case, \(n\) is a least witness.
\end{proof}

\begin{theorem}\label{thm:catclassical}
Let \(c\) be an arbitrary additive congruence on \((\N,+,0)\). Then either
\(c\) is equality, or
\(
c=\equiv_{i,p}
\)
for a unique pair \(i\geq0\), \(p\geq1\). The derivation is conditional on
\(\mathrm{EM}\).
\end{theorem}

{\begin{proof}
By \(\mathrm{EM}\), either
\[
\forall x,y\in\N,\qquad x\sim y\longrightarrow x=y,
\]
or this statement fails. In the first case, \(c\) is equality.

In the second case, \(\mathrm{EM}\) applied to the existence of a related pair
of distinct elements gives \(u\neq v\) with \(u\sim v\).
Without loss of generality, assume \(u<v\).
Then \(v-u\) is a positive period at \(u\).
Lemma~\ref{lem:lnp}, applied to the predicate
\[
q\longmapsto 0<q\ \wedge\ u\sim u+q,
\]
gives the least positive period \(g\) at \(u\). Applied to
\[
n\longmapsto n\sim n+g,
\]
it gives the least point \(i\) at which \(g\) is a period.
Proposition~\ref{prop:core} gives
\(
c=\equiv_{i,g},
\)
and Lemma~\ref{lem:ipunique} gives uniqueness.
\end{proof}}

The isomorphism
\(
D\colon(\Ndel,+,0)\longrightarrow(\N,+,0)
\)
induces a bijection between additive congruences on \(\Ndel\) and additive
congruences on \(\N\). Moreover,
\[
x\equiv_{i,p}y
\quad\Longleftrightarrow\quad
D(x)\equiv_{i,p}D(y).
\]
Hence the three classification theorems for \((\N,+,0)\) are equivalent to
their corresponding forms for \((\Ndel,+,0)\). The classical form is
Theorem~\ref{thm:catalogue}.

\subsection{Exactness of the prices}\label{ssec:exactness}

Theorems~\ref{thm:catMarkov's} and~\ref{thm:catclassical} show that
\(\mathrm{MP}\) and \(\mathrm{EM}\) are sufficient. We now prove that they
are also necessary for the corresponding classification.

Let \(\mathrm{MarkovForm}\) denote the statement that every decidable
additive congruence on \((\N,+,0)\) which is not equality is equal to
\(\equiv_{i,p}\) for a unique pair \(i\geq0\), \(p\geq1\).

Let \(\mathrm{ClassicalForm}\) denote the statement that every additive
congruence on \((\N,+,0)\) is either equality or is equal to
\(\equiv_{i,p}\) for a unique pair \(i\geq0\), \(p\geq1\).

\begin{proposition}\label{prop:mpsharp}
The statement \(\mathrm{MarkovForm}\) implies Markov's principle.
\end{proposition}

\begin{proof}
Let \(\theta\) be a decidable predicate on \(\N\), and suppose
\(
\neg\neg\exists n,\theta(n).
\)
We define
\[
x\sim_\theta y
\quad:\Longleftrightarrow\quad
x=y
\quad\vee\quad
\exists k\leq\min(x,y),\theta(k).
\]

We first check that \(\sim_\theta\) is an additive congruence. Reflexivity
and symmetry are immediate. For transitivity, suppose
\(
x\sim_\theta y, y\sim_\theta z.
\)
If \(x=y\) or \(y=z\), then \(x\sim_\theta z\) follows directly. Assume
that \(x\neq y\) and \(y\neq z\). Then there exist
\[
k\leq\min(x,y),\qquad \ell\leq\min(y,z)
\]
such that \(\theta(k)\) and \(\theta(\ell)\). Since order on \(\N\) is
decidable, either \(x\leq z\) or \(z<x\). If \(x\leq z\), then
\[
k\leq x=\min(x,z).
\]
If \(z<x\), then
\[
\ell\leq z=\min(x,z).
\]
Thus \(x\sim_\theta z\).

For compatibility with addition, let \(x\sim_\theta y\) and \(t\in\N\). If
\(x=y\), then \(x+t=y+t\). Otherwise the relation is witnessed by some
\(k\leq\min(x,y)\), and then
\[
k\leq\min(x+t,y+t),
\]
so \(x+t\sim_\theta y+t\). The relation is decidable because equality is
decidable and the search for \(k\) is bounded.

{Suppose that \(\sim_\theta\) is equality. For any \(k\), if \(\theta(k)\)
holds, then
\[
k\sim_\theta k+1,
\]
although \(k\neq k+1\). Hence \(\neg\theta(k)\) for every \(k\), and therefore
\(
\neg\exists k\,\theta(k).
\)
This contradicts the assumption
\(
\neg\neg\exists k\,\theta(k).
\)
Thus \(\sim_\theta\) is not equality.}

By \(\mathrm{MarkovForm}\), there exist \(i\geq0\) and \(p\geq1\) such that
\(
\sim_\theta\;=\;\equiv_{i,p}.
\) {Since
\(
i\equiv_{i,p}i+p,
\)
we have \(i\sim_\theta i+p\). Since \(p\geq1\), we have \(i\neq i+p\), so
the equality assumption is impossible. Therefore
\[
\exists k\leq\min(i,i+p)\,\theta(k).
\]
Since \(\min(i,i+p)=i\), it follows that
\(
\exists k\,\theta(k).
\)
Since \(\theta\) is arbitrary, then \(\mathrm{MP}\) follows.}
\end{proof}

\begin{proposition}\label{prop:emsharp}
The statement \(\mathrm{ClassicalForm}\) implies EM.
\end{proposition}

\begin{proof}
Let \(P\) be an arbitrary proposition and define
\[
x\sim_P y
\quad:\Longleftrightarrow\quad
x=y\quad\vee\quad P.
\]
This is an additive congruence. Reflexivity and symmetry are immediate.
Transitivity follows by cases, and if \(x\sim_P y\), then
\[
x+t=y+t\quad\vee\quad P,
\]
so \(x+t\sim_P y+t\). By \(\mathrm{ClassicalForm}\), either \(\sim_P\) is equality or
\(
\sim_P\;=\;\equiv_{i,p}
\)
for some \(i\geq0\), \(p\geq1\).

In the first case, \(P\) implies
\(
0\sim_P1,
\)
and hence \(0=1\). Therefore \(\neg P\).

In the second case,
\(
i\equiv_{i,p}i+p,
\)
and therefore
\(
i\sim_P i+p.
\)
Thus
\[
i=i+p\quad\vee\quad P.
\]
Since \(p\geq1\), we have \(i\neq i+p\). Hence \(P\).

Therefore \(P\vee\neg P\). Since \(P\) is arbitrary, then EM follows.
\end{proof}

{Theorems~\ref{thm:catMarkov's} and~\ref{thm:catclassical}, together with
Propositions~\ref{prop:mpsharp} and~\ref{prop:emsharp}, give the following
equivalences.

\begin{corollary}\label{cor:exact}
The two classification forms satisfy
\[
\mathrm{MarkovForm}\Longleftrightarrow\mathrm{MP},
\qquad
\mathrm{ClassicalForm}\Longleftrightarrow\mathrm{EM}.
\]
Thus the {metatheoretic logical} prices \(\{\mathrm{MP}\}\) and \(\{\mathrm{EM}\}\) are exact.
\end{corollary}

\begin{proof}
Theorem~\ref{thm:catMarkov's} gives
\[
\mathrm{MP}\Longrightarrow\mathrm{MarkovForm},
\]
and Proposition~\ref{prop:mpsharp} gives the converse. Similarly,
Theorem~\ref{thm:catclassical} gives
\[
\mathrm{EM}\Longrightarrow\mathrm{ClassicalForm},
\]
and Proposition~\ref{prop:emsharp} gives the converse.
\end{proof}}

Together with Theorem~\ref{thm:catforced}, this completes the proof of
Theorem~\ref{thm:pricing}.
\begin{proposition}\label{prop:decidableEM}
Assume \(\mathrm{EM}\). Then every additive congruence on \((\N,+,0)\) is
equal to a decidable relation.
\end{proposition}

\begin{proof}
Let \(c\) be an additive congruence. By
Theorem~\ref{thm:catclassical}, either \(c\) is equality or
\(
c=\equiv_{i,p}
\)
for some \(i\geq0\) and \(p\geq1\). Equality on \(\N\) is decidable, and
\(\equiv_{i,p}\) is decidable by Lemma~\ref{lem:indexperiod}. Hence \(c\) is equal
to a decidable relation.
\end{proof}}

{\begin{proposition}\label{prop:decsharp}
Let \(\mathrm{MarkovForm}^{-}\) denote the statement obtained from
\(\mathrm{MarkovForm}\) by deleting the decidability hypothesis: every
additive congruence on \((\N,+,0)\) which is not equality is equal to
\(\equiv_{i,p}\) for a unique pair \(i\geq0\), \(p\geq1\). Then
\(\mathrm{MarkovForm}^{-}\) implies \(\mathrm{EM}\).
\end{proposition}

\begin{proof}
Let \(P\) be a proposition, and let \(\sim_P\) be the additive congruence from
the proof of Proposition~\ref{prop:emsharp}, defined by
\[
x\sim_P y\quad:\Longleftrightarrow\quad x=y\ \vee\ P.
\]
Assume \(\neg\neg P\). If \(\sim_P\) is equality, then \(P\)  implies
\(0\sim_P1\), and hence \(0=1\). Therefore \(\neg P\), contradicting
\(\neg\neg P\). Thus \(\sim_P\) is not equality.

By \(\mathrm{MarkovForm}^{-}\), there are \(i\geq0\) and \(p\geq1\) such that
\(
\sim_P\;=\;\equiv_{i,p}.
\)
Since \(i\equiv_{i,p}i+p\), we have \(i\sim_P i+p\), and therefore
\[
i=i+p\quad\vee\quad P.
\]
Since \(p\geq1\), the equality \(i=i+p\) is impossible. Hence \(P\). We have
therefore proved
\[
\neg\neg P\longrightarrow P
\]
for every proposition \(P\). Applied to \(P\vee\neg P\), whose double negation
is provable
intuitionistically, this gives \(\mathrm{EM}\).\end{proof}}

{The relation \(\sim_P\) is not assumed decidable. Indeed, a decision procedure
for \(\sim_P\) implies \(P\), since
\(
0\sim_P1\Longleftrightarrow P.
\) 
Hence \(\mathrm{MarkovForm}\) does not apply.

\begin{corollary}\label{cor:vacuous}
Under \(\mathrm{EM}\), the decidability hypothesis of
Theorem~\ref{thm:catMarkov's} excludes no congruence. However, it cannot be
removed from the Markov form:
\[
\mathrm{MarkovForm}^{-}\Longleftrightarrow\mathrm{EM},
\qquad
\mathrm{MarkovForm}\Longleftrightarrow\mathrm{MP}.
\]
Moreover, \(\mathrm{MP}\) does not imply \(\mathrm{EM}\)~\cite{ishihara}.
\end{corollary}

\begin{proof}
The first statement follows from Proposition~\ref{prop:decidableEM}.
Proposition~\ref{prop:decsharp} gives
\[
\mathrm{MarkovForm}^{-}\Longrightarrow\mathrm{EM}.
\]
Conversely, assume \(\mathrm{EM}\), and let \(c\) be an additive congruence
which is not equality. By Theorem~\ref{thm:catclassical}, we have
\(
c=\equiv_{i,p}
\)
for some \(i\geq0\) and \(p\geq1\), and uniqueness follows from
Lemma~\ref{lem:ipunique}. Thus
\[
\mathrm{EM}\Longrightarrow\mathrm{MarkovForm}^{-}.
\]
The second equivalence is Corollary~\ref{cor:exact}. The strict separation
between \(\mathrm{MP}\) and \(\mathrm{EM}\) is given in~\cite{ishihara}.
\end{proof}}

For the classification  considered here, the exact prices are
\(
\varnothing,\{\mathrm{MP}\},\{\mathrm{EM}\}.
\)
No separate \(\mathrm{LPO}\)-level occurs.

{\section{Lean formalization}\label{sec:lean}

The main results are formalized in Lean~4~\cite{lean4} over
Mathlib~\cite{mathlib}. The formalization used here is contained in the repository
\texttt{actual-mathematics}%
\footnote{\url{https://github.com/jonwashburn/actual-mathematics}}. The declarations used in this paper
are in the module families
\texttt{ActualMathematics.DeltaKernel} and
\texttt{ActualMathematics.DeltaForced}.
{The rigidity declarations belong to the module family
\texttt{ActualMathematics.Rigidity}.}

The command \texttt{\#print axioms} is less precise than the ledger of
Section~\ref{ssec:ledger}. It records the axioms used by a Lean declaration,
but not the logical principles used inside a derivation. Table~\ref{tab:lean}
lists the corresponding Lean declarations.

\begin{table}[ht]
\caption{Main results and their Lean~4 declarations.}
\label{tab:lean}
\centering
\footnotesize
\begin{tabular}{@{}>{\raggedright\arraybackslash}p{0.34\textwidth}>{\raggedright\arraybackslash}p{0.60\textwidth}@{}}
\toprule
\textbf{Result} & \textbf{Lean declaration} \\
\midrule
Theorem~\ref{thm:initiality} & \lean{base\_initial} \\
Theorem~\ref{thm:embedding} & \lean{baseRec\_injective} \\
Theorem~\ref{thm:rigidity} & \lean{base\_categorical},
  \lean{base\_unique\_iso}, \lean{base\_rigidity} \\Proposition~\ref{prop} & \lean{bootstrap\_initiality} \\
Theorem~\ref{thm:sound} & \lean{sound\_forced} \\
Theorem~\ref{thm:construction} & \lean{bootstrap\_tower\_export},
  \lean{delta\_bootstrap} \\
Proposition~\ref{thm:standardencodable} & \lean{deltaForced\_nat},
  \lean{deltaForced\_int}, \lean{deltaForced\_rat} \\
Theorem~\ref{thm:catalogue} & \lean{orbit\_congruence\_catalogue} \\
Theorem~\ref{thm:recognition} & \lean{recognizer\_dichotomy} \\
Proposition~\ref{prop:core} & \lean{catalogue\_core} \\
Theorem~\ref{thm:catforced} & \lean{nat\_catalogue\_forced},
  \lean{orbit\_catalogue\_forced} \\
Theorem~\ref{thm:catMarkov's} & \lean{nat\_catalogue\_markov},
  \lean{orbit\_catalogue\_markov} \\
Theorem~\ref{thm:catclassical} & \lean{nat\_catalogue\_classical},
  \lean{orbit\_catalogue\_classical} \\
Proposition~\ref{prop:mpsharp} & \lean{markov\_of\_catalogue} \\
Proposition~\ref{prop:emsharp} & \lean{em\_of\_catalogue} \\
Corollary~\ref{cor:exact} & \lean{catalogue\_markov\_iff\_mp},
  \lean{catalogue\_classical\_iff\_em} \\
\bottomrule
\end{tabular}
\end{table}

Table~\ref{tab:lean} contains  the main results. The auxiliary
results, including the semiring laws and the finite index-period monoids,
are formalized in the same module families.

At the audited commit, the declarations checked by the audit manifest,\\
except \lean{orbit\_congruence\_catalogue} and
\lean{recognizer\_dichotomy}, have axiom sets contained in
\(
\{\texttt{propext},\ \texttt{Quot.sound}\}.
\)
The two exceptions also use \texttt{Classical.choice}.
They give the
classical form of the classification, which by
Corollary~\ref{cor:exact} is equivalent to \(\mathrm{EM}\). In the module\\
\texttt{DeltaKernel.RecognitionQuotientsLedger}, the required logical
principle is included as a hypothesis.  These versions do not use
\texttt{Classical.choice}. For fixed \(i\) and \(p\), the congruence
\lean{ipCon} and the quotient \(M(i,p)\) are choice-free.

The declarations \lean{base\_initial}, \lean{baseRec\_injective},
\lean{base\_categorical}, \lean{base\_unique\_iso}, and\\
\lean{base\_rigidity} use neither \texttt{Classical.choice} nor excluded
middle. The declaration \lean{base\_categorical} depends on no axioms.
The declaration \lean{baseRec\_injective} is stated for a Peano
realization, but its proof uses only injectivity of the step and the
condition that the base point is not a successor.

\begin{remark}\label{rem:invisible}
 The axiom audit does not distinguish the two prices separated in
Section~\ref{ssec:exactness}. In Lean,  EM and arbitrary
propositional decidability both appear as a dependency on
\texttt{Classical.choice}. In the ledger formalization, the required
principle is included as a hypothesis, so the two prices remain separate.
\end{remark}}

\section{Conclusion}\label{sec:conclusion}

We separated a distinction witness from the generated \(\delta\)-orbit. A
distinction witness gives two unequal elements, but it gives no operation which
can be iterated. The \(\delta\)-orbit is instead generated from the empty record
\(0\) by the one-step extension \(S\).  Its inductive definition yields
successor nonzero, successor injectivity, and induction.  Together with the recursively defined operations $+$ and $\cdot$, it presents the
natural-number object \(\Ndel\).

{Moreover, $\Ndel$ is initial among $\delta$-algebras. If the step is
injective and the base point is not in its image, the canonical
homomorphism from $\Ndel$ is injective. Hence every such realization
contains a canonical copy of $\Ndel$. If the realization also satisfies
induction, it is uniquely isomorphic to $\Ndel$. Thus the Peano conditions
characterize the generated $\delta$-orbit; they do not produce a different
arithmetic object.}

Arithmetic, in the form of \(\Ndel\), is thus the canonical object generated by iterated distinction. The Peano conditions characterize this object up to a unique isomorphism rather than provide a second construction of the natural numbers. Each element of the orbit records a finite number of successive distinction steps, and the later stages of the number tower are constructed explicitly from these records by quotient constructions. {The classification of recognition quotients carries an exact metatheoretic logical cost}, and the reverse implications show that the Markov and excluded-middle prices are exact.

The corresponding \(\delta\)-calculus is an intuitionistic first-order proof
system over the signature \(\{0,S,+,\cdot\}\). Its term algebra is free on the de Bruijn variables. Each derivation carries a ledger, and the forced fragment is sound in the standard model \(\mathfrak N\).

{We then constructed the choice-free tower
\(
\Ndel\hookrightarrow\Zdel\hookrightarrow\Qdel.
\)
The metatheoretic number systems $\N$, $\Z$, and $\Q$ are
$\delta$-encodable. The tower constructions and the explicit encodings
use no choice principle. The continuum remains the next unresolved boundary, since the
Cauchy and Dedekind constructions need not be equivalent constructively.}


We also classified the recognition quotients of the additive monoid
\((\Ndel,+,0)\). Assuming \(\mathrm{EM}\), every recognizer
\(
r\colon(\Ndel,+,0)\longrightarrow(V,\ast,e)
\)
is either injective or has kernel congruence \(\equiv_{i,p}\) for a unique
pair \(i\geq0\), \(p\geq1\). In the second case, its recognition quotient is
isomorphic to the finite monogenic monoid \(M(i,p)\). 


Finally, the classification of the recognition quotients admits three forms.
A decidable congruence together with an explicit pair of distinct related elements implies the classification without additional nonconstructive principles. For a decidable congruence which is
not equality, it is conditional on \(\mathrm{MP}\). For an arbitrary
congruence, the dichotomy between equality and an index-period congruence is
conditional on \(\mathrm{EM}\). The reverse implications show that
 \(\mathrm{MP}\) and \(\mathrm{EM}\) are also necessary, so these
metatheoretic logical prices are exact.  The decidability assumption is classically redundant, but
constructively it separates the Markov form from the excluded-middle form.
This distinction is not visible to the axiom audit of the host system.

\bigskip

\vspace{6pt}
\noindent {\bf Author Contributions:} {Conceptualization, J.W.; Methodology, J.W. and M.Z.; Software, J.W.; Validation, J.W. and M.Z.; Formal Analysis, M.Z. and J.W.; Investigation, J.W. and M.Z.; Resources,
J.W.; Writing---Original Draft Preparation, J.W.; Writing---Review and Editing, M.Z. and J.W.; Funding Acquisition, J.W. All authors have read and agreed to the published version of the~manuscript.}

\end{document}